\documentclass[11pt, a4paper, reqno]{amsart}
\DeclareMathAlphabet{\mathscrbf}{OMS}{mdugm}{b}{n}

\usepackage[utf8]{inputenc}
\usepackage[slovene, english]{babel}
\usepackage{polski}
\usepackage{amsmath,amsfonts,amssymb,mathrsfs,amsthm,amscd, ulem, stmaryrd}
\usepackage[pdftex]{graphicx}
\usepackage{tikz-cd}
\usepackage{soul} 
\usepackage{MnSymbol}
\usepackage[all,cmtip]{xy}
\usepackage{enumitem}

\usepackage{slashed}

\usepackage{pgfplots}

\usepackage[mathcal]{euscript}  

\usepackage{url}

\usetikzlibrary{arrows}

 \usepackage{relsize}

\tikzset{
  no line/.style={draw=none,
    commutative diagrams/every label/.append style={/tikz/auto=false}},
  from/.style args={#1 to #2}{to path={(#1)--(#2)\tikztonodes}}}

\title[Stone duality between condensed mathematics and algebraic geometry]{
Stone duality between condensed mathematics \\ and algebraic geometry}
\author{Rok Gregoric}
\thanks{Johns Hopkins University}
\date{\today}
\address{Johns Hopkins University, Baltimore, MD 21218, USA}
\email{rgrego12@jhu.edu}
\newtheorem{theoremi}{Theorem}
\newtheorem{theorem}{Theorem}[section]
\newtheorem{quest}[theorem]{Question}

\newtheorem{corollary}[theorem]{Corollary}
\newtheorem{lemma}[theorem]{Lemma}
\newtheorem{prop}[theorem]{Proposition}

\theoremstyle{definition}
\newtheorem{definition}[theorem]{Definition}

\newtheorem{cons}[theorem]{Construction}

\newtheorem{exun}[theorem]{Example}
\newtheorem{remark}[theorem]{Remark}

\usepackage{tikz, calc}
\usetikzlibrary{matrix,arrows}

\newcommand*{\CAlg}{{\operatorname{CAlg}}}
\newcommand*{\CRing}{{\operatorname{CRing}}}

\newcommand*{\mC}{\mathcal C}
\newcommand*{\mD}{\mathcal D}

\newcommand*{\mU}{\mathcal U}

\newcommand*{\mP}{\mathcal P}
\newcommand*{\mX}{\mathcal X}
\newcommand*{\mY}{\mathcal Y}
\newcommand*{\mZ}{\mathcal Z}
\newcommand*{\mS}{\mathcal S}

\newcommand*{\sF}{\mathscr F}

\newcommand*{\sheafhom}{\mathscr{H}\kern -.5pt om}
\DeclareMathOperator{\Novak}{\mathscr{N}\text{\kern -3pt {\calligra\large ovak}}\,\,}

\usepackage{amsmath,calligra,mathrsfs}
\DeclareMathOperator{\fHom}{\mathscr{H}\text{\kern -3pt {\calligra\large om}}\,}
\DeclareMathOperator{\id}{\operatorname{id}}

\DeclareMathOperator{\Hom}{\operatorname{Hom}}

\DeclareMathOperator{\Fun}{\operatorname{Fun}}

\DeclareMathOperator{\Spec}{\operatorname{Spec}}

\DeclareMathOperator{\QCoh}{\operatorname{QCoh}}

\DeclareMathOperator{\Mod}{\operatorname{Mod}}

\renewcommand{\i}{\infty}

\renewcommand{\i}{\infty}

\usepackage{color}   
\usepackage[hypertexnames=false]{hyperref}
\hypersetup{
    linktoc=all,     
    linkcolor=black,  
}

\renewcommand{\i}{\infty}
\renewcommand{\o}{\otimes}

\usepackage{epigraph}

\newcommand{\stackspace}{3.5}
\newcommand{\stack}[2][1cm]{\;\tikz[baseline, yshift=.65ex]%
    {\foreach \k [evaluate=\k as \r using (.5*#2+.5-\k)*\stackspace] in {1,...,#2}{%
    \ifodd\k{\draw[<-](0,\r pt)--(#1,\r pt);}%
    \else{\draw[->](0,\r pt)--(#1,\r pt);}\fi
    }}\;}

\pgfplotsset{compat=1.18}
\begin{document}

\begin{abstract}
We extend Stone duality to a fully faithful embedding of condensed sets into fpqc sheaves over an arbitrary field, which preserves colimits and finite limits. We study how familiar notions from condensed mathematics/topology and algebraic geometry correspond to each other under this form of Stone duality.
\end{abstract}

\maketitle

\epigraph{
\textit{As it turned out, [algebraic geometry] seems to have acquired
the reputation of being esoteric, exclusive and very abstract with adherents who
are secretly plotting to take over all the rest of mathematics!}
}{D.~Mumford, \cite{Mumford}}



\renewcommand*{\thetheoremi}{\Alph{theoremi}}

The main result of this paper is that topology, in the guise of condensed mathematics, can be fully and faithfully embedded into algebraic geometry by a form of Stone duality. The corresponding algebro-geometric spaces are not as pathological as one might expect, and their geometric properties reflect many topological properties on the other side.\\

Condensed mathematics, pioneered by  Barwick-Haine \cite{Pyknotic}, Clausen-Scholze \cite{Condensed}, and anticipated by Lurie \cite{Ultra}, is roughly an alternative to point-set topology. Its chief advantage is that it interacts better and more robustly with algebraic and categorical constructions.
Up to set-theoretic nuances, a condensed set is defined to be a sheaf on the category of profinite sets $\mathrm{ProFin}$. The connection to topology comes from the identification of profinite sets with totally disconnected compact Hausdorff topological spaces.

On the other hand, the most celebrated result concerning profinite sets is \textit{Stone duality}, first obtained in \cite{Stone on duality}. It may be expressed as a contravariant equivalence of categories
\begin{equation}\label{Classic Stone}
\mathrm{ProFin}^\mathrm{op}\simeq \mathrm{CAlg}_{\mathbf F_2}^\mathrm{Bool}
\end{equation}
between profinite sets and \textit{Boolean rings}: rings $A$ in which $a=a^2$ holds for every $a\in A$. The latter are equivalent to Boolean algebras of lattice theory and logic, but we prefer to adopt the perspective of viewing them as special kinds of rings.
In this paper, we extend Stone duality to relate condensed mathematics and algebraic geometry.
\begin{theoremi}[Stone duality for condensed sets; $\mathbf F_2$-version]\label{Main thm F_2}
The equivalence of categories \eqref{Classic Stone} induces an
 adjunction between condensed sets and fpqc sheaves over $\mathbf F_2$
$$
\mathrm{Cond}\rightleftarrows \mathrm{Shv}^\mathrm{fpqc}_{\mathbf F_2},
$$
whose left adjoint constituent preserves all finite limits and is fully faithful.
\end{theoremi}

The most salient aspect of Theorem \ref{Main thm F_2} is the fully faithfulness assertion. It means that condensed sets can always be understood as fpqc sheaves. This process of passing from topology/condensed mathematics to algebraic geometry loses no information and reflects a surprising amount of properties and structure into familiar algebro-geometric terms:

\begin{theoremi}[Theorems \ref{condensed is geometric}, \ref{Shv = QCoh},
Propositions \ref{non-algebraic}, \ref{Global functions}, \ref{Closed prop}, \ref{Open prop}, Corollary \ref{points are points}]
Stone duality gives a dictionary between condensed mathematics and algebraic geometry, under which:
\begin{center}
\begin{tabular}{| |c|| c| | } \hline \hline
\textbf{condensed set $X$} &  \textbf{fpqc sheaf $\mX$} \\
\hline \hline
$X$ is a profinite set &$\mX$ is an affine scheme\\
\hline
$X$ is a compact Hausdorff space & $\mX$ is a geometric space \\ 
\hline
$X$ is a CGWH space & $\mX$ is an ind-geometric space \\ 
\hline
$X$ is an arbitrary condensed set & $\mX$ is a quotient of an ind-geometric space  \\   & by an ind-geometric equivalence relation  \\ 
\hline
$X$ is a discrete condensed set & $\mX$ is an algebraic space loc.\,of fin.\,pres. \\
\hline
open/closed embeddings into $X$ & open/closed immersions into $\mX$\\
\hline
the underlying set $X(*)$ & the set of $\mathbf F_2$-points $\mX(\mathbf F_2)$\\
\hline
$\Hom_{\mathrm{Cond}}(X, \mathbf F_2)$ & $\Gamma(\mX; \CMcal O_{\mX})$\\
\hline
$\mathrm{Shv}_{\Mod_{\mathbf F_2}}(X)$ & $\QCoh(\mX)$ \\
\hline
\end{tabular} 
\end{center}
\noindent where a \textit{geometric space} is roughly an fpqc-version of an algebraic space.
\end{theoremi}

\vspace{0.26em}

\noindent At this point, it is natural to ask whether there is something special about the field $\mathbf F_2$ that makes it possible to embed condensed mathematics into algebraic geometry over it. We show this is not the case. On the contrary, our results remain valid -- and are proved  -- over an arbitrary field $k$, with the added assumption that $k$ is finite only required at a few select places.

\begin{theoremi}[Stone duality for condensed sets; $k$-version -- Theorem \ref{FF on topoi}]\label{Main thm k}
Let $k$ be a field. There is a canonical
 adjunction between condensed sets and fpqc sheaves over $k$
$$
\mathrm{Cond}\rightleftarrows \mathrm{Shv}^\mathrm{fpqc}_{k},
$$
whose left adjoint constituent preserves all finite limits and is fully faithful.
\end{theoremi}

The adjoint functors of Theorem \ref{Main thm k} admit the following descriptions:
\begin{itemize}
\item (Propositions \ref{FOP for Haus}, \ref{Stone space in terms of pi_0}) The left adjoint sends a condensed set $X:\mathrm{ProFin}^\mathrm{op}\to\mathrm{Set}$ to the fpqc sheafification of the functor
$$
\CAlg_k\ni A\mapsto X(\pi_0(\left|\Spec(A)\right|))\in \mathrm{Set},
$$
where the connected components of the underlying set $\left |\Spec(A)\right|$ are equipped with the quotient (in particular, totally disconnected) topology. For compactly generated Hausdorff topological spaces $X$ this simplifies to
$$
\CAlg_k\ni A\mapsto \Hom_\mathrm{Top}(\left |\Spec(A)\right |, X)\in \mathrm{Set}.
$$
\item (Proposition \ref{How enh works}) The right adjoint functor of Theorem \ref{Main thm k} sends an fpqc sheaf $\mX :\CAlg_k\to\mathrm{Set}$ into the condensed space
$$
\mathrm{ProFin}^\mathrm{op}\ni S\mapsto \mX(\CMcal C(S, k))\in \mathrm{Set},
$$
where $\CMcal C(S, k)$ denotes the ring of continuous functions $S\to k$.
\end{itemize}

 For a compact Hausdorff space $X$, viewed as a qcqs condensed set, the corresponding fpqc sheaf $\mX:\CAlg_k\to \mathrm{Set}$ can be described in terms of sheaves of $k$-modules.  Using a recent version \cite{German} of Tannaka reconstruction, we express it (Proposition \ref{Tannaka rec}) as
$$
\mX(A)\simeq \mathrm{LFun}^{\otimes}_{k}(\mathrm{Shv}_{\Mod_k}(X), \Mod_A),
$$
which denotes the $k$-linear colimit-preserving symmetric monoidal functors.
Note that this is precisely the ``Tannakian spectrum" of the sheaf category $\mathrm{Shv}_{\Mod_k}(X)$.
When $X$ is a profinite set, this can  be further decategorified (Proposition \ref{Stone space in terms of pi_0}) to
$$
\mX(A)\simeq \Hom_{\CAlg_k}(\CMcal C(X, k), A).
$$
We furthermore show (Proposition \ref{Affinization is pi_0}) that affinization of $\mX$ in general corresponds to the appropriately-interpreted profinite completion of the condensed set $X$.

By restricting Theorem \ref{Main thm k} to the subcategories of profinite sets inside condensed sets on the left and of affine $k$-schemes on the right-hand side respectively, we recover a $k$-algebra version of Stone duality 
$$
\mathrm{ProFin}^\mathrm{op}\simeq 
\CAlg_k^\mathrm{Stone}
$$
due to \cite{Antieau}, which specializes to \eqref{Classic Stone} for $k=\mathbf F_2$. Here $\CAlg_k^\mathrm{Stone}\subseteq\CAlg_k$ denotes the full subcategory of \textit{Stone $k$-algebras}, which form the backbone of our approach to Theorem \ref{Main thm k}.
Over a separably closed field $k$, Stone $k$-algebras coincide with weakly \'etale $k$-algebras. It follows that in that setting, Theorem \ref{Main thm k}  recovers the instigating approach to condensed mathematics via the pro-\'etale site of \cite{Bhatt-Scholze}.

\subsection*{Summary of contents}

This paper is organized as follows:
\begin{itemize}
\item Section \ref{Section 1} discusses Stone $k$-algebras and establishes the algebraic prerequisites concerning them required in the proof of the main theorem.
\item Section \ref{Section 2} reviews the construction of condensed sets and fpqc sheaves from \cite{Condensed} and 
\cite{Dirac2}, then proves the main comparison theorem.
\item Section \ref{Stone duality section} investigates how topological notions correspond to algebro-geometric ones by the Stone duality result of the previous section.
\item Section \ref{Appendix} shows how the results of this paper also work in the ``light condensed'' setting of \cite{CS23}, and discusses a possible higher-categorical generalization.
\end{itemize}

\subsection*{Acknowledgments}
I am tremendously grateful to Ben Antieau, David Ben-Zvi, Anish Chedalavada, David Gepner, Peter Haine, Deven Manam, Arpon Raksit, and Alberto San Miguel Malaney for useful conversations on the topic of this paper. I am also grateful to Emily Riehl for organizing a seminar on condensed mathematics at Johns Hopkins in the Fall of 2023, which indirectly inspired this work.

\section{Preliminaries on Stone algebras}\label{Section 1}

In this paper, all rings and algebras are commutative and unital. Let $\CAlg_k$  denote the category of (always implicitly small) $k$-algebras. In this section, we will primarily be occupied with its 
 following subcategory, introduced in \cite[Definition 3.2]{Antieau}.

\begin{definition}\label{Def of Stone}
Fix a base ring $k$.
The \textit{category of Stone $k$-algebras} $\mathrm{CAlg}_k^\mathrm{Stone}\subseteq \CAlg_k$ is the full subcategory of commutative $k$-algebras, generated under filtered colimits by the  $k$-algebras $k^S = \bigoplus_S k$ for finite sets $S$.
\end{definition}

We establish a form of Stone duality for Stone algebras in Subsection \ref{Subsection 1.1}. In Subsection \ref{Subsection 1.2} we define the \textit{pearl} of a $k$-algebra, the maximal Stone subalgebra, and discuss its algebraic property. Then we connect it to connected components of spectral spaces in Subsection \ref{Subsection 1.3}. Subsection \ref{Subsection 1.4} shows that Stone duality is an equivalence of sites. Finally, we identify the class of Stone $k$-algebras with weakly \'etale $k$-algebras when $k$ is separably closed, and with $p$-Boolean algebras when $k=\mathbf F_p$ in Subsection \ref{Subsection 1.5}.

\subsection{Stone duality}\label{Subsection 1.1}

The \textit{raison d'\'etre} of Stone $k$-algebras is that they participate in a duality with profinite sets. When $k=\mathbf F_2$, this recovers the usual notion of Stone duality, as originated in \cite{Stone on duality}, upon recognizing Stone $\mathbf F_2$-algebras as Boolean rings and identifying those with Boolean algebras as in \cite{Stone on rings}. The following version of Stone duality is proved in {\cite[Lemma 3.5]{Antieau}}, but since the precise functors giving rise to this equivalence are so important to this paper, we provide a detailed proof as well.

\begin{theorem}[Stone duality]\label{Stone theorem}
Let $k$ be a ring whose only idempotent elements are $0, 1\in k$.
There is a canonical equivalence of categories
$$
\CAlg_k^\mathrm{Stone}\simeq \mathrm{ProFin}^\mathrm{op},
$$
given for any Stone $k$-algebra $A$ and profinite set $S$ by
$$
A\mapsto \Hom_{\CAlg_k^\mathrm{Stone}}(A, k),\qquad\qquad S\mapsto \Hom_{\mathrm{Cond}}(S, k) \simeq \CMcal C(S, k).
$$
The Hom sets here inherit the respective profinite and $k$-algebra structure from that of $k$.
\end{theorem}

\begin{proof}

The full subcategory $\CAlg_k^\mathrm{min}\subseteq \CAlg_k$, spanned by the  $k$-algebras $k^S$ for $S\in\mathrm{Fin}$, consists of compact objects and is closed under finite colimits and retracts. It follows that $\CAlg_k^\mathrm{Stone}$ is presentable and compactly generated by the objects of $\CAlg_k^\mathrm{min}$. The functor $\mathrm{Fin}^\mathrm{op}\to \CAlg_k^\mathrm{min}$, given by $S\mapsto k^S$, extends to a functor $\mathrm{ProFin}^\mathrm{op}\to \CAlg_k^\mathrm{Stone}$,  given by
$$
S=\varprojlim_i S_i\mapsto \varinjlim_i k^{S_i} \simeq \CMcal C (S, k)
$$
by the standard identification of the continuous functions on a profinite set. Since both the discrete set $k$ and the profinite set $S$ belong to the class of compactly generated weak Hausdorff spaces, which embed fully faithfully into condensed sets, we may equivalently replace the continuous maps $\CMcal C$ above with condensed-set maps $\Hom_\mathrm{Cond}$.

The so-obtained functor $\mathrm{ProFin}^\mathrm{op}\to \CAlg_k^\mathrm{Stone}$ is essentially surjective by construction. We wish to show it is fully faithful as well. By the definition of Stone $k$-algebras, their Hom spaces can be described as
\begin{eqnarray*}
\Hom_{\CAlg_k^\mathrm{Stone}}(A, B) &=& 
\Hom_{\CAlg_k}(\varinjlim_i k^{S_i}, \varinjlim_j k^{S_j})\\
&=& 
\varprojlim_i\varinjlim_j\Hom_{\CAlg_k}( k^{S_i}, k^{S_j})\\
&=& 
\varprojlim_i\varinjlim_j\prod_{S_j}\Hom_{\CAlg_k}( k^{S_i}, k),
\end{eqnarray*}
in which the second equality makes use of the $k$-algebra $k^{S_i}$ being finitely generated and as such a compact object in the category $\CAlg_k$. 
It follows that the functor in question being fully faithful boils down to the canonical map $S\to \Hom_{\CAlg_k}(k^S, k)$ being a bijection for any finite set $S$. To verify that, assume without loss of generality that $S=\{0, 1, \ldots, n\}$. If $n\ge 2$, we can write this as
$$
S=\{0, 1\}\times_{\{1\}}\cdots\times_{\{n-1\}}\{n-1, n\}\simeq \{0, 1\}^{\times n-2}.
$$
The functor $S\mapsto k^S$ takes products of finite sets to coproducts of $k$-algebras i.e.\ tensor products, so
\begin{eqnarray*}
\Hom_{\CAlg_k}(k^S, k) &=& \Hom_{\CAlg_k}((k^{\{0, 1\}})^{\otimes n-2}, k)\\
&=& \Hom_{\CAlg_k}(k^{\{0, 1\}}, k)^{\times n-2}\\
&=& \Hom_{\CRing}(\mathbf Z[t]/(t^2-1), k)^{\times n-2}\\
&=& \mathrm{Idem}(k)^{\times n-2},
\end{eqnarray*}
where we have used the identification of rings $\mathbf Z^{\{0, 1\}} = \mathbf Z\oplus\mathbf Z\simeq \mathbf Z[t]/(t^2-1)$ and maps out of it into a ring $k$ with idempotents in $k$. Since $\mathrm{Idem}(k)=\{0, 1\}$ by assumption, we find that
$$
\Hom_{\CAlg_k}(k^S, k) = \{0, 1\}^{\times n-2} = S
$$
as desired, proving that the functor in question $\mathrm{ProFin}^\mathrm{op}\to \CAlg_k^\mathrm{Stone}$ is fully faithful and therefore an equivalence of categories.
It follows from the above proof of full faithfulness that an inverse functor $\CAlg_k^\mathrm{Stone}\to\mathrm{ProFin}^\mathrm{op}$ can be given by $A\mapsto \Hom_{\CAlg_k}(A, k)$. 
\end{proof}

\begin{remark}
The assumption that $0, 1\in k$ are the only idempotent elements of the ring $k$ is equivalent to requiring that the underlying topological space $\left |\Spec(k)\right |$ of its spectrum is connected. For this reason, such a ring is sometimes called a \textit{connected ring}, especially in commutative algebra. This class contains all fields, as well as all local rings and integral domains, including the universal base $\mathbf Z$.
\end{remark}

\subsection{The pearl of an algebra}\label{Subsection 1.2}

Recall that a ring is called \textit{absolutely flat} (also known as \textit{von Neumann regular}, especially in non-commutative contexts) if every module over it is flat.
The corresponding relative notion is that of a \textit{weakly \'etale algebra} in the sense of \cite[\href{https://stacks.math.columbia.edu/tag/092B}{Definition 092B}]{stacks-project}. An $A$-module is flat if and only if it is flat as a $k$-module for any weakly \'etale $k$-algebra $A$. 

\begin{prop}\label{Is wetal}
Any Stone $k$-algebra is a weakly \'etale $k$-algebra.
\end{prop}

\begin{proof}
Stone $k$-algebras are by definition filtered colimits of the finite \'etale $k$-algebras $k^S$ for finite sets $S$, so it follows from \cite[\href{https://stacks.math.columbia.edu/tag/092N}{Lemma 092N}]{stacks-project} that all Stone $k$-algebras are weakly \'etale $k$-algebras in the sense of \cite[\href{https://stacks.math.columbia.edu/tag/092B}{Definition 092B}]{stacks-project}.
\end{proof}

\begin{prop}\label{Stone implies abs flat}
If $k$ is absolutely flat, then so is any Stone $k$-algebra $A$.
\end{prop}

\begin{proof}
This follows from the preceding Proposition
because weakly \'etale $k$-algebras over an absolutely flat ring $k$ are absolutely flat
by \cite[\href{https://stacks.math.columbia.edu/tag/092I}{Lemma 092I}]{stacks-project}.
\end{proof}

\begin{prop}\label{Pearl exists}
The forgetful functor $U :\CAlg_k^\mathrm{Stone}\to \CAlg_k$ admits a right adjoint.
\end{prop}

\begin{proof}
According to the first isomorphism theorem, it suffices to note that $\CAlg_k^\mathrm{Stone}$ is presentable (as we saw in the proof of Theorem \ref{Stone theorem}), and that the forgetful functor $U$ respects colimits. Since any colimit can obtained from finite and filtered colimits, it suffices to show that the subcategory of Stone $k$-algebras is closed under both of those types of colimits. For filtered colimits, this follows directly from the definition. For finite colimits, which are computed in $\CAlg_k$ as relative tensor products, this ultimately boils down to tensor products commuting with colimits and the observation that $k^S\otimes_k k^{S'}=k^{S\times S'}$ holds for all finite sets $S, S'$.
\end{proof}

\begin{definition}
The \textit{pearl of a $k$-algebra $A$}, denoted $A^\circ$, is its image under the right adjoint functor $\CAlg_k\to\CAlg_k^\mathrm{Stone}$ of Proposition \ref{Pearl exists}. That is to say,  $A^\circ\to A$ is terminal among $k$-algebra homomorphisms from Stone $k$-algebras into $A$.
\end{definition}

We wish to show that the pearl $A^\circ$ may be viewed as a subring of $A$. That will require a well-known basic result about absolutely flat rings.

\begin{lemma}\label{ideals in abs flat}
Any finitely generated ideal $I$ in an absolutely flat ring $R$ is of the form $I=(e)$ for some idempotent element $e\in R$.
\end{lemma}

\begin{proof}
Consider the short exact sequence of $R$-modules
$$0\to I\to R\to R/I\to 0.$$
By tensoring with $I$, and using the fact that it is a flat $R$-module, we find that $I^2=I$. Thus any ideal in an absolutely flat ring is idempotent.
If $I$ is principal, so that $I=(a)$ for some $a\in R$, this means that $a=a^2b$ holds for some $b\in R$. The element $e=ab$ then generates the same ideal $I$ and satisfies $e^2=(a^2b)b=e$, hence it is an idempotent generator.
It remains to show that any finitely generated ideal is principal. By induction, it suffices to restrict to the case of two generators $I=(a, b)$.
Each of the principal ideals $(a)$ and $(b)$ can be generated by idempotents, so it suffices to assume that $a, b\in R$ are idempotent elements. Now consider the element 
$e=a+b-ab\in I$. We have $ae=a^2+ab-a^2b = a$ and likewise $be=b$, thus $I=(e)$ as desired.
\end{proof}

\begin{prop}\label{Pearl as a subring}
Let $k$ be absolutely flat and $A$ a $k$-algebra. The universal map from the pearl $A^\circ \to A$ is injective.
\end{prop}

\begin{proof}
It suffices to show that if $f:B\to A$ is a $k$-algebra map from a Stone $k$-algebra $B$, then the $k$-algebra $\mathrm{im}(f)$ is Stone as well. By the first isomorphism theorem, it further suffices to show that Stone $k$-algebras are closed under quotients.  For any ideal in a ring $I\subseteq R$, we may write the quotient as a filtered colimit $R/I\simeq \varinjlim_J R/J$ indexed on finitely generated subideals $J\subseteq I$. Since Stone $k$-algebras are closed under filtered colimits, it is therefore enough to show that they are closed under forming quotients by finitely generated ideals. By Proposition \ref{Stone implies abs flat} and Lemma \ref{ideals in abs flat}, this further reduces to considering the case of idempotent-generated principal ideals.
Recall that an idempotent  element $e\in A$ induces an $A$-linear splitting $A = eA\oplus (1-e)A$. It follows that $A/e\simeq (1-e)A$, and so we are left with showing that a $k$-subalgebra $B\subseteq A$ of a Stone $k$-algebra $A$ is Stone as well. Indeed, if we write a Stone $k$-algebra as a filtered colimit
$$
A=\varinjlim_{i\,\in\, \mathcal I} k^{S_i}
$$
for finite sets $S_i$, then the $k$-subalgebra $B$ may be written as
$$
B=\varinjlim_{i\,\in \,\mathcal I_B} k^{S_i},
$$
where $\mathcal I_B\subseteq\mathcal I$ is the full subcategory of all $i\in \mathcal I$ for which the image of the corresponding map $k^{S_i}\to A$ factors through the inclusion $B\subseteq A$. Since we have for all  indices $i, j$
$$
k^{S_i}\otimes_k k^{S_j} = k^{S_i\times S_j},
$$
 the index category $\mathcal I_B\subseteq \mathcal I$ is a filtered subcategory, thus exhibiting $B$ as a Stone $k$-algebra as desired.
\end{proof}

A consequence of the preceding results is the following alternative description of the Stone duality functor. The need for it to hold is perhaps the primary reason why we have to work over a base field (as opposed to say a connected base ring) throughout this paper.

\begin{prop}\label{Stone as underspace}
If $k$ is a field, the Stone duality functor $\CAlg_k^\mathrm{Stone}\simeq \mathrm{ProFin}^\mathrm{op}$ from Theorem \ref{Stone theorem}
is given by
$$
A\mapsto \left |\Spec(A)\right |,
$$
the underlying topological space of the spectrum of the Stone $k$-algebra $A$.
\end{prop}

\begin{proof}
We know that the functor in question
 can be described as $A\mapsto \Hom_{\CAlg_k}(A, k)$, inheriting its profinite topology from $k$. There is an obvious inclusion
 \begin{equation}\label{k-points}
 \Hom_{\CAlg_k}(A, k)\subseteq \left |\Spec(A)\right|,
 \end{equation}
 where the right-hand side is identified with the $k$-points of $\Spec(A)$.
Conversely, suppose $x\in \left |\Spec(A)\right|$ is any point. This corresponds to a prime ideal in the ring $A$, which is absolutely flat Proposition \ref{Stone implies abs flat}. Since any prime ideal in an absolutely flat ring is maximal  \cite[\href{https://stacks.math.columbia.edu/tag/092F}{Lemma 092F}]{stacks-project}, this is a closed point, and therefore determined by the $k$-algebra map $A\to \kappa(x)$ into the residue field. We saw in the proof of \ref{Pearl as a subring} that Stone $k$-algebras are closed under the formation of quotients. Thus $\kappa(x)$ is both a Stone $k$-algebra and a field; but that is only possible if $\kappa(x)=k$ itself. Thus $x$ corresponds to a $k$-algebra map $A\to k$ and so the subspace inclusion \eqref{k-points} is an equality.
\end{proof}

\subsection{The pearl and connected components}\label{Subsection 1.3}
The pearl of a $k$-algebra can be expressed without explicitly referencing Stone $k$-algebras. Since a $k$-algebra $A$ is Stone precisely when the canonical $k$-algebra homomorphism from the pearl $A^\circ\to A$ is an isomorphism, this in fact characterizes Stone $k$-algebras.

\begin{cons}\label{Cons of pi_0}
Let $\pi_0(X)$ denote\footnote{Note that $\pi_0(X)$ is used here, in accordance with the conventions of \cite{stacks-project}, for the set of connected components, as opposed to its more common usage for the set of \textit{path-connected} components.} the set of \textit{set of connected components} of a topological space $X$, equipped with the quotient topology. It is totally disconnected by  \cite[\href{https://stacks.math.columbia.edu/tag/08ZL}{Lemma 08ZL}]{stacks-project}, and the resulting functor $\pi_0: \mathrm{Top}\to \mathrm{Top}^\mathrm{tot.disc.}$ is left adjoint to the full subcategory inclusion $\mathrm{Top}^\mathrm{tot.disc.}\subseteq\mathrm{Top}$.
Restricting to the full subcategory $\mathrm{Top}^\mathrm{sp}\subseteq\mathrm{Top}$ of spectral spaces in the sense of \cite[\href{https://stacks.math.columbia.edu/tag/08YG}{Definition 08YG}]{stacks-project}, then $\pi_0(X)$ is profinite by  \cite[\href{https://stacks.math.columbia.edu/tag/0906}{Lemma 0906}]{stacks-project}. Thus $\pi_0:\mathrm{Top}^\mathrm{sp}\to \mathrm{ProFin}$ is left adjoint to the inclusion $\mathrm{ProFin}\to \mathrm{Top}^\mathrm{sp}$ which identifies profinite sets with totally disconnected spectral spaces.
\end{cons}

\begin{prop}\label{Pearl as pi_0}
Let $k$ be a field and $A$ a $k$-algebra. The universal map of affine schemes $\Spec(A)\to \Spec(A^\circ)$ exhibits a homeomorphism of underlying topological spaces
$$
\left|\Spec(A^\circ)\right|\simeq \pi_0(\left|\Spec(A)\right|).
$$ 
Here the set of connected components $\pi_0(\left|\Spec(A)\right|)$ is equipped with the inherited quotient topology from $\left|\Spec(A)\right|$, which is totally disconnected i.e.\ profinite. 
\end{prop}

\begin{proof}

Fix a profinite set $S$. There is a chain of natural bijections
\begin{eqnarray*}
\Hom_{\mathrm{ProFin}}(\left|\Spec(A^\circ)\right |, S) &\simeq &\Hom_{\CAlg^\mathrm{Stone}_k}(\CMcal C(S, k), A^\circ)\\
&\simeq &\Hom_{\CAlg_k}(\CMcal C(S, k), A)\\
&\simeq &\Hom_{\mathrm{Sch}_k}(\Spec(A), \Spec(\CMcal C(S, k)))\\
&\simeq &\Hom_\mathrm{Top}(\left|\Spec(A)\right |, S)\\
&\simeq &\Hom_{\mathrm{ProFin}}(\pi_0\left (|\Spec(A)\right |), S) 
\end{eqnarray*}
in which the first bijection is Stone duality, the second is the universal property of the pearl functor, the third is the usual equivalence between affine $k$-schemes and $k$-algebras, the discussion of the fourth bijection we postpone for a moment, and the last bijection is the adjunction from Construction \ref{Cons of pi_0}. Since $S$ was an arbitrary profinite set, the desired result follows by the Yoneda lemma.

It remains to justify the fourth bijection in the above chain. That is to say, we wish to establish a natural bijection
$$
\Hom_{\mathrm{Sch}_k}(\mX, \Spec(\CMcal C(S, k)))\simeq \Hom_{\mathrm{Top}}(\left |\mX\right|, S)
$$
for any affine scheme $\mX =\Spec(A)$. Writing the profinite set as a filtered limit $S=\varprojlim_i S_i$ of finite sets $S_i$, we have a $k$-scheme isomorphism
$$
\Spec(\CMcal C(S, k))\simeq \varprojlim_i \coprod_{S_i}\Spec(k),
$$
from which it follows that  the underlying topological space is given by
$$
\left |\Spec(\CMcal C(S, k))\right | \simeq S\times \left |\Spec(k)\right |
$$
as topological spaces over $\left |\Spec(k)\right |$. The passage  $\mY\mapsto \left |\mY\right|$ to the underlying topological space of a scheme  therefore gives a natural map
$$
\Hom_{\mathrm{Sch}_k}(\mX, \Spec(\CMcal C(S, k)))\to \Hom_{\mathrm{Top}_{/\left |\Spec(k)\right|}}(\left |\mX\right |, S\times \left |\Spec(k)\right |)\simeq  \Hom_{\mathrm{Top}}(\left |\mX\right |, S)
$$
as desired, which we wish to show is a bijection. Since both sides, as well as the construction of the map between them, commute with limits in $S$, we reduce to the case where $S$ is a finite set, let us say with $n$ elements. The left-hand side is equivalent to the set of $k$-algebra homomorphisms $k^n\to A$, i.e.\ an $n$-tuple of pair-wise orthogonal idempotents $e_1, \ldots, e_n\in A$ (i.e.\ $e_ie_j=\delta_{ij}$). The right-hand side consists of continuous maps $f:\left|\mX\right|\to\{1, \ldots, n\}$, i.e.\ a partition of $\left|\mX\right|$ into $n$ disjoint clopen subsets $f^{-1}(\{i\})$. Under the usual correspondence $f^{-1}(\{i\}) = V(e_i)$ between clopens in $\mX$ and idempotents in $A$ of \cite[\href{https://stacks.math.columbia.edu/tag/00EB}{Section 00EB}]{stacks-project}, we obtain the desired bijection.
\end{proof}

\begin{corollary}\label{Pearl as functions}
The pearl of a $k$-algebra $A$ over a field $k$ can be expressed as the algebra of continuous functions
$
A^\circ = \CMcal C(\left |\Spec(A)\right|, k).
$
\end{corollary}

\begin{proof}
By Proposition \ref{Pearl as pi_0} and Stone duality in the form of Theorem \ref{Stone theorem}, the pearl of $A$ is given by the formula $
A^\circ = \CMcal C(\pi_0(\left |\Spec(A)\right|), k).
$
Since the discrete set $k$ is totally disconnected, we have $\CMcal C(\left |\Spec(A)\right|, k)\simeq \CMcal C(\pi_0\left |\Spec(A)\right|, k)$ by the adjointness of connected components from Construction \ref{Cons of pi_0}.
\end{proof}

\begin{exun}\label{Example extension}
Let $L/k$ be any field extension. Using Corollary \ref{Pearl as functions}, we may express the pearl of the $k$-algebra $L$ as
$$
L^\circ = \CMcal C(\left|\Spec(L)\right|, k) = \CMcal C(*, k) = k.
$$
In particular, a $k$-algebra $A$ can fail to be Stone when $\left|\Spec(A)\right |$ is totally disconnected (i.e.\ equal to $\pi_0(\left |\Spec(A)\right|)$, even if $A$ is additionally assumed to be reduced.
\end{exun}

\begin{exun}\label{Idem no prod}
Continuing with the previous example, let $L/k$ be a finite Galois extension. Then there is a canonical $k$-algebra isomorphism
$$
L\otimes_k L\simeq \prod_{\mathrm{Gal}(L/k)} L.
$$
Since taking pearls preserves products, we find that the pearl of this tensor product is given by
$$
(L\o_k L)^\circ \simeq \prod_{\mathrm{Gal}(L/k)} L^\circ =  \prod_{\mathrm{Gal}(L/k)} k
$$
On the other hand, we have $L^\circ\o_k L^\circ  = k\o_k k = k$, showing that in this case the injective map $L^\circ\o_k L^\circ \to (L\o_k L)^\circ$ is not bijective whenever $L/k$ is a non-trivial extension.
\end{exun}

\subsection{Stone algebras and the fpqc topology}\label{Subsection 1.4}

We will need some auxiliary results on how Stone $k$-algebras interact with the fpqc topology on $k$-algebras.

\begin{lemma}\label{Lemma ff for Stone}
A $k$-algebra homomorphism $A\to B$ between Stone $k$-algebras is faithfully flat if and only if the induced continuous map $\left|\Spec(B)\right |\to\left|\Spec(A)\right|$ is surjective.
\end{lemma}

\begin{proof}
It is standard \cite[\href{https://stacks.math.columbia.edu/tag/00HQ}{Lemma 00HQ}]{stacks-project} that a flat ring map is faithfully flat if and only if it induces a surjection on spectra. To prove the desired assertion, it therefore suffices to show that any $k$-algebra homomorphism between Stone $k$-algebras is flat. Since Stone $k$-algebras are weakly \'etale by Proposition \ref{Is wetal},  it follows from \cite[\href{https://stacks.math.columbia.edu/tag/092L}{Lemma 092L}]{stacks-project} that all $k$-algebra map between them are weakly \'etale themselves, and as such flat.
\end{proof}

\begin{lemma}\label{FF for Stone}
Let $k$ be absolutely flat and let $A$ be a Stone $k$-algebra. A $k$-algebra map $A\to B$ 
is faithfully flat if and only if it is injective.
\end{lemma}

\begin{proof}

Recall from \cite[Lemma 5.5]{Lurie's Tannaka} that a ring map $A\to B$ is faithfully flat if and only if it is injective and both $B$ and $B/A$ are flat $A$-modules. Since Stone $k$-algebras are absolutely flat by Proposition \ref{Stone implies abs flat}, the flatness requirements are always satisfied.  
\end{proof}

\begin{prop}\label{site preservation lemma}
The forgetful functor $U :\CAlg_k^\mathrm{Stone}\to \CAlg_k$ preserves fpqc covers. If $k$ is absolutely flat, then so does its right adjoint $(-)^\circ : \CAlg_k\to \CAlg_k^\mathrm{Stone}$.
\end{prop}

\begin{proof}
The claim for the forgetful functor is immediate.
For the statement about pearl functor, note that $A\to B$ is any faithfully flat morphism of $k$-algebras, it is in particular injective. The induced morphism of pearls $A^\circ\to B^\circ$ can thanks to Proposition \ref{Pearl as a subring} be viewed as map between subrings of $A$ and $B$. From this perspective, it is clearly injective as well, and therefore faithfully flat by Lemma \ref{FF for Stone}.
\end{proof}

The pearl functor does not in general preserve tensor products, as we will see in Example \ref{Idem no prod}.  Nonetheless, there is a sufficiently well-behaved comparison map.

\begin{prop}\label{Pearl and products}
Let $k$ be absolutely flat, $A$ a Stone $k$-algebra, and $A\to B$ a faithfully flat $k$-algebra map.
The maps $B\rightrightarrows B\otimes_A B$
 induce a canonical faithfully flat map on pearls
$$
B^\circ\otimes_A B^\circ\to (B\otimes_AB)^\circ.
$$
\end{prop}

\begin{proof}
The ring maps $B\rightrightarrows B\otimes_AB$ give rise to a $k$-algebra map $B^\circ\o_A B^\circ \to B\o_A B$. The domain $B^\circ\o_A B^\circ$ is a quotient of the Stone $k$-algebra $B^\circ\o_k B^\circ$, and we saw in the proof of Proposition \ref{Pearl as a subring} that Stone algebras are closed under quotients. Therefore $B^\circ \o_A B^\circ$ is a Stone $k$-algebra as well. It follows from the universal property of the pearl that the $k$-algebra  $B^\circ\o_A B^\circ \to B\o_A B$ factors through the inclusion $(B\o_A B)^\circ \subseteq B\o_A B$.
To show that the induced ring map is injective, and therefore faithfully flat by Lemma \ref{FF for Stone}, it suffices by Proposition \ref{Pearl as a subring} to prove that $B^\circ \o_A B^\circ \to B\o_A B$ is.
But since $A\to B$ is faithfully flat, the morphism in question factors as a composite of injections
$$
B^\circ\o_A B^\circ \to B^\circ \otimes_A B\to B\otimes_A B
$$
where we twice applied  that $B^\circ\to B$ is faithfully flat, as follows from combining Proposition \ref{Pearl as a subring} and Lemma \ref{FF for Stone}.
\end{proof}

In the foundations of condensed mathematics, profinite sets are equipped with the \textit{effective epimorphism topology}.
 That means that a collection of maps $\{f_i : X_i\to Y\}_{i\in I}$ in $\mathrm{ProFin}$ is a cover if and only if a finite subset of it $\{f_{i_1}, \ldots, f_{i_n}\}$ is jointly surjective, i.e.\ if $f_{i_1}(X_1)\cup\cdots f_{i_n}(X_n)=Y.$ This finitely-jointly-surjective condition is reminiscent of the analogous requirement in the definition of the fpqc topology on affines. As we now show, this is no coincidence.

\begin{prop}\label{Prop1}
Let $k$ be a field.
Stone duality of Theorem \ref{Stone theorem} induces an isomorphism of sites
$$
(\mathrm{ProFin}, \tau_{\mathrm{eff}})\simeq ((\CAlg_k^\mathrm{Stone})^\mathrm{op}, \tau_\mathrm{fpqc})
$$
between profinite sets with the effective epimorphism topology and the opposite of Stone $k$-algebras with the fpqc topology, inherited from the inclusion  $\CAlg_k^\mathrm{Stone}\subseteq \CAlg_k$.
\end{prop}

\begin{proof}
We need to show that the Stone duality equivalence of categories identifies covers of profinite sets in the effective epimorphism topology with fpqc covers of Stone $k$-algebras. By the definitions of both topologies in question, it suffices to do so for covers consisting of a single element. This follows directly from Proposition \ref{Stone as underspace} and Lemma \ref{Lemma ff for Stone}.
\end{proof}

The hypothesis in Proposition \ref{Prop1} that the ground ring $k$ is a field is there to ensure that both the Stone duality of Theorem \ref{Stone theorem} and the preceding results concerning faithful flatness can be applied. In fact, this requirement singles out fields among all rings.

\begin{lemma}\label{what are fields}
A ring $k$ is both absolutely flat and has no idempotent elements other than $0, 1\in k$ if and only if it is a field.
\end{lemma}

\begin{proof}
That clearly holds for a field. Conversely, if $k$ is absolutely flat, then by Lemma \ref{ideals in abs flat}, each of its principal ideals is generated by an idempotent element. Since $0, 1\in k$ are the only idempotents, we have for any $a\in k$ either $(a)=(1)=k$, in which case $a$ is invertible, or $(a)=(0)$, in which case $a=0$. That is to say, $k$ is a field.
\end{proof}

\subsection{Stone algebras over fields}\label{Subsection 1.5}

Any Stone algebra is weakly \'etale by Lemma \ref{Is wetal}.
 The question of when the converse holds, at least over a field, is answered by the following.

\begin{prop}Let $k$ be a field.
The  inclusion of categories
$\CAlg_k^\mathrm{Stone}\subseteq\CAlg_k^\mathrm{w\acute{e}t}$ is the identity  if and only if the field $k$ is separably closed.
\end{prop}

\begin{proof}
A finite \'etale $k$-algebra is a product of rings $L_1\oplus\cdots\oplus L_n$ of finite separable field extensions $L_i/k$. If $k$ is separable, then we have $L_i = k$ for all $i$, showing that every finite \'etale $k$-algebra is of the form $k^S$ for some $S\in \mathrm{Fin}$. Over a field $k$, weakly \'etale $k$-algebras can by \cite[\href{https://stacks.math.columbia.edu/tag/092Q}{Lemma 092Q}]{stacks-project} all be described as filtered colimits of finite \'etale $k$-algebras, showing that they agree with Stone $k$-algebras.

Conversely, suppose that every weakly \'etale $k$-algebra is Stone. Any finite separable field extension $L/k$ is weakly \'etale, which implies that $L^\circ = L$. But the calculation from Example \ref{Example extension} shows that $L^\circ = k$. It follows that $L=k$ is the only finite separable field extension, which means that $k$ is separably closed.
\end{proof}

\begin{remark}
The pearl of an algebra $A$ over a separably closed field $k$ can be identified with the maximal weakly \'etale subalgebra $A^\circ = B_\mathrm{max}(A)$ of \cite[\href{https://stacks.math.columbia.edu/tag/0CKS}{Lemma 0CKS}]{stacks-project}.
\end{remark}

\begin{corollary}\label{proetale site}
Let $k$ be a field.
The site $((\CAlg_k^\mathrm{Stone})^\mathrm{op}, \tau_\mathrm{fpqc})$ is equal to the pro-\'etale $\Spec(k)_\mathrm{pro\acute{e}t}$ of \cite{Bhatt-Scholze}
if and only if $k$ is separably closed.
\end{corollary}

\begin{remark}
Corollary \ref{proetale site} shows that the Stone $k$-algebra approach to condensed sets, stemming from Proposition \ref{Prop1} and expanded upon in the next section, recovers over a separable field the original pro-\'etale topology approach in which condensed sets are the sheaves on the pro-\'etale site of a point. This accords with the fact that the spectrum of a field $\Spec(k)$ is really only a point from the \'etale perspective if $k$ is separably closed. 
\end{remark}

Another family of fields over which we can identify Stone algebras with a known class of algebras are the prime fields $\mathbf F_p$ for prime numbers $p$.
Recall that an $\mathbf F_p$-algebra $A$ is called \textit{$p$-Boolean} (classically also \textit{$p$-ring}) if any element $a\in A$ satisfies $a^p=a$. At  $p=2$, this recovers the notion of \textit{Boolean rings}, which are canonically equivalent to the Boolean algebras of lattice theory and logic by  \cite{Stone on rings}.

\begin{prop}\label{Stone is Bool}
Stone $\mathbf F_p$-algebras coincide with $p$-Boolean algebras for any prime $p$.
\end{prop}

\begin{proof}
According to Corollary \ref{Pearl as functions}, an $\mathbf F_p$-algebra $A$ is Stone precisely when the canonical $k$-algebra map $\CMcal C(\left |\Spec(A)\right|, \mathbf F_p)\to A$ is an isomorphism. It therefore remains to identify the continuous functions $\left| \Spec(A)\right|\to \mathbf F_p$ with the elements $a\in A$ which satisfy $a^p=a$. Since the codomain of a function $\left| \Spec(A)\right|\to \mathbf F_p$ is discrete, its fibers (i.e.\ the preimages of points) determine a partition of the topological space $\left |\Spec(A)\right|$ into $p$ clopen subsets. The latter correspond to idempotent elements $e_1, \ldots, e_p\in A$ which are pair-wise orthogonal in the sense that $e_ie_j =0$ holds for all $i\ne j$, and which satisfy $e_1+\cdots + e_p = 1$. Since the last idempotent $e_p$ is automatically determined by this equation as $e_p = 1-(e_1+\cdots + e_{p-1})$, this is equivalent to the collection of pair-wise idempotents $e_1, \ldots, e_{p-1}\in A$. By \cite[Theorem 1, Corollary 1]{Zemmer: p-rings and their Boolean geometry} this corresponds precisely to an element $a\in A$ for which $a^p=a$, with the correspondence given by
$$
a = e_1 + 2e_2 +\cdots + (p-1)e_{p-1}, \qquad\qquad e_i = 1-(a-i)^p,\quad  i=1, \ldots, p-1.
$$
In particular, the canonical map $\CMcal C(\left |\Spec(A)\right |, \mathbf F_p)\to A$ is given in terms of this presentation by $a\mapsto a$. It follows that an $\mathbf F_p$-algebra $A$ is Stone if and only if $a^p=a$ for all elements $a\in A$, i.e.\ if it is $p$-Boolean.
\end{proof}

\begin{remark}
By the proof of Proposition \ref{Stone is Bool}, the pearl of an $\mathbf F_p$-algebra $A$ is given by
$$
A^\circ = \{a\in A : a^p=a\},
$$
which may equivalently be seen as the fixed points of the Frobenius automorphism $A^{\varphi =1}$. For $p=2$, this is precisely the subring of idempotent elements.
\end{remark}

\begin{corollary}\label{left adjoint Fp}
Let $p$ be a prime number. The forgetful functor $U: \CAlg^\mathrm{Stone}_{\mathbf F_p}\to \CAlg_{\mathbf F_p}$ admits a left adjoint $Q:\CAlg_{\mathbf F_p}\to\CAlg_{\mathbf F_p}^\mathrm{Stone}$.
\end{corollary}

\begin{proof}
For any $\mathbf F_p$-algebra $A$, consider the quotient construction
$$
Q(A) \coloneq A/(a^p-a: a\in A).
$$
This is a $p$-Boolean $\mathbf F_p$-algebra by construction, and any $\mathbf F_p$-algebra morphism $A\to B$ into a $p$-Boolean ring $B$ factors uniquely through the quotient map $A\to Q(A)$. In light of the identification between Stone $\mathbf F_p$-algebras and $p$-Boolean rings of Proposition \ref{Stone is Bool}, this therefore gives the desired left adjoint.   
\end{proof}

The way we obtained the left adjoint $Q$ in the proof of Corollary \ref{left adjoint Fp} crucially relies on the alternative description of Stone $\mathbf F_p$-algebras from Proposition \ref{Stone is Bool}. In its absence, we are unaware of the existence of such a left adjoint for other base fields $k$.

\begin{quest}
For which fields $k$ does the forgetful functor $U: \CAlg^\mathrm{Stone}_k\to \CAlg_k$ admit a left adjoint?
\end{quest}

\section{Comparison between condensed sets and fpqc sheaves}\label{Section 2}

In this section, we use the theory of Stone algebras and Stone duality developed in the previous section to compare condensed sets and algebraic geometry. The latter must be understood here in a ``functor of points" way, and with respect to the fpqc (as opposed to the Zariski or \'etale) topology.

We begin in Subsection \ref{Subsection 2.1} with a detailed review of how to set up the foundations of both this perspective on algebraic geometry, as well as condensed mathematics, being mindful of the set-theoretical nuances this entails. Building on that, we prove our main theorem in Subsection \ref{Subsection 2.2}. In Subsection \ref{Subsection 2.3} we unpack and examine the functors furnished by the main theorem, preparing for their more detailed in study in the next section.

\subsection{The setup of condensed sets and fpqc sheaves}\label{Subsection 2.1}
Let us fix some notation. We let  $\mP(\mC)\coloneq\mathrm{Fun}(\mC^\mathrm{op}, \mathrm{Set})$ denote the category of presheaves on a category $\mC$. If $\mC$ is equipped with a Grothendieck topology $\tau$, then $\mathrm{Shv}(\mC, \tau)\subseteq \mP(\mC)$ will denote the category of sheaves on the site $(\mC, \tau)$. When the site $(\mC, \tau)$ is small, the sheaf category $\mathrm{Shv}(\mC, \tau)$ is a Grothendieck topos.

The sites $(\mathrm{ProFin}, \tau_\mathrm{eff})$ and $(\mathrm{CAlg}_k^\mathrm{op}, \tau_\mathrm{fpqc})$, of relevance to condensed mathematics and algebraic geometry respectively, are both non-small\footnote{The issue of non-smallness is \textit{\`a priori} present for all standard sites in algebraic geometry, e.g\ Zariski, \'etale, fppf, etc. But unlike those, the fpqc site (just as the profinite site) does not even admit a small set of generators. If it did, we could restrict to the subcategory generated by those without affecting the sheaves, which is the standard way of circumventing these set-theoretic issues in algebraic geometry.}
 Therefore, taking sheaves on them is not guaranteed to lead to well-behaved categories - in particular, they fail to be Grothendieck topoi, and the existence of sheafification functors becomes dubious. These issues are merely of a set-theoretical (more precisely, size-theoretic) nature though, and not conceptual. In this subsection outline the approach to dealing with these issues as developed in the condensed setting by Clausen-Scholze in \cite{Condensed}, and extended to the fpqc setting in  \cite{Dirac2} by Hesselholt-Pstrągowski.

\begin{definition}\label{Def of topoi}
Let $\kappa$ be an uncountable strong limit cardinal and $k$ a $\kappa$-small ring.
The categories of \textit{$\kappa$-condensed sets} and of \textit{$\kappa$-fpqc sheaves over $k$} are the categories of  sheaves
$$
\mathrm{Cond}_\kappa \coloneq {\mathrm{Shv}}(\mathrm{ProFin}_\kappa, \tau_{\mathrm{eff}}), \qquad\quad \mathrm{Shv}^{\mathrm{fpqc}}_{k, \kappa} \coloneq \mathrm{Shv}(\CAlg_{k, \kappa}^\mathrm{op}, \tau_\mathrm{fpqc})
$$
on the respective sites of all $\kappa$-small profinite sets $\mathrm{ProFin}_\kappa$ equipped with the effective epimorphism topology, and the opposite category of  $\kappa$-small $k$-algebras $\CAlg_{k, \kappa}$  equipped with the fpqc topology.
\end{definition}

To dispense with the auxiliary cardinal bound $\kappa$, we make use of the following result, the fpqc-part of which is a variant of a result of Waterhouse \cite{Waterhouse}.

\begin{prop}[{\cite[Proposition 2.9]{Condensed}, \cite[Lemma 12.15, Lemma A.8]{Dirac2}}]\label{up the beanstalk}
Let $\kappa'>\kappa$ be strongly inaccessible cardinals. The left Kan extensions along the subcategory inclusions $\mathrm{ProFin}_\kappa\subseteq \mathrm{ProFin}_{\kappa'}$ and $\CAlg_{k, \kappa}^\mathrm{op}\subseteq \CAlg_{k, \kappa'}^\mathrm{op}$ restrict to canonical functors
$$
\mathrm{Cond}_\kappa\to\mathrm{Cond}_{\kappa'}, \qquad \qquad \quad
\mathrm{Shv}_{k, \kappa}^{\mathrm{fpqc}}\to \mathrm{Shv}_{k, \kappa'}^{\mathrm{fpqc}}
$$
which are fully faithful and commute with all colimits and all $\kappa$-small limits.
\end{prop}

\begin{definition}\label{Def of big cats}
Let $k$ be a ring. The categories of \textit{condensed sets} and \textit{fpqc sheaves over $k$} are the filtered colimits of categories
$$
\mathrm{Cond} \coloneq\varinjlim_\kappa \mathrm{Cond}_\kappa,\qquad\qquad \mathrm{Shv}_k^\mathrm{fpqc} \coloneq\varinjlim_\kappa \mathrm{Shv}_{k, \kappa}^\mathrm{fpqc}
$$
over the filtered poset of strongly inaccessible cardinals.
\end{definition}

\begin{remark}
That is to say, a condensed set of fpqc sheaf is always left Kan extended from a $\kappa$-version for some strongly inaccessible cardinal $\kappa$, so for any practical considerations, we may work in the Grothendieck topoi $\mathrm{Cond}_\kappa$ and $\mathrm{Shv}_{k,\kappa}^\mathrm{fpqc}$, even though the entire categories $\mathrm{Cond}$ and $\mathrm{Shv}_k^\mathrm{fpqc}$ are too big to be topoi. This is purposefully analogous to the usual way of avoiding set-theoretic issues with the category of sets $\mathrm{Set}$ in practice.
\end{remark}

The categories of Definition \ref{Def of big cats} are by \cite[Footnote 5]{Condensed} and \cite[Remark 2.16]{Dirac2} equivalent to the full subcategories of the categories of sheaves on the respective big sites
$$
\mathrm{Cond}\subseteq \mathrm{Shv}(\mathrm{ProFin}, \tau_\mathrm{eff}), \qquad \quad \mathrm{Shv}^\mathrm{fpqc}_k\subseteq \mathrm{Shv}(\CAlg_k^\mathrm{op}, \tau_\mathrm{fpqc}),
$$
spanned by those sheaves whose underlying presheaves are accessible.
Let $\mP_\mathrm{acc}(\mC)\subseteq \mP(\mC)$ denote the subcategory  of accessible presheaves on a not-necessarily-small category $\mC$. By the argument of \cite[Corollary 2.17]{Dirac2}, the inclusions of subcategories of sheaves into accessible presheaves extend to adjunctions
$$
(-)^\dagger:\mP_\mathrm{acc}(\mathrm{ProFin})\rightleftarrows \mathrm{Cond}, \qquad \quad (-)^\dagger:\mP_\mathrm{acc}(\CAlg_k^\mathrm{op})\rightleftarrows \mathrm{Shv}^\mathrm{fpqc}_k.
$$
We abusively refer to the left adjoints $(-)^\dagger$ as sheafification. This is in part justified by the following remark.

\begin{remark}\label{compatibility of sheafification}
Let $\kappa$ be a strongly inaccessible cardinal.
By precomposing the tentative sheafification functors discussed above with the canonical fully faithful embeddings to their respective $\kappa$-variants from Definition \ref{Def of topoi}, we find by unwinding the definitions that we recover the genuine sheafification. That is to say, the diagrams of categories
$$
\begin{tikzcd}
\mP(\mathrm{ProFin}_\kappa)\ar{r}{(-)^\dagger}\ar{d} & \mathrm{Cond}_\kappa \ar[d]\\
\mP_\mathrm{acc}(\mathrm{ProFin}) \ar{r}{(-)^\dagger} & \mathrm{Cond}
\end{tikzcd} \qquad\quad
\begin{tikzcd}
\mP(\CAlg_{k,\kappa}^\mathrm{op})\ar{r}{(-)^\dagger}\ar[d] & \mathrm{Shv}^\mathrm{fpqc}_{k,\kappa} \ar[d]\\
\mP_\mathrm{acc}(\CAlg_k^\mathrm{op}) \ar{r}{(-)^\dagger} & \mathrm{Shv}^\mathrm{fpqc}_{k}
\end{tikzcd}
$$
commute, where the upper horizontal arrows are the topos-level sheafifications and the vertical arrows are given by left Kan extension. 
\end{remark}

\subsection{The main theorem}\label{Subsection 2.2}
Having recalled the setup for both condensed sets and fpqc sheaves, we arrive at our main result: extending Stone duality to a fully faithful embedding of condensed sets into fpqc sheaves. So far as we are aware, this seems to be a new observation even in the case of a separably closed field $k=k^s$ where it by Proposition \ref{proetale site} concerns sheaves on the pro-\'etale site.

\begin{theorem}\label{FF on topoi}
Let $k$ be a field.
The composite of Stone duality of Theorem \ref{Stone theorem} and the subcategory inclusion
$$
\mathrm{ProFin}\simeq (\CAlg^\mathrm{Stone}_{k})^\mathrm{op}\subseteq \CAlg_{k}^\mathrm{op}
$$
gives rise to an adjunction
$$
f^* :\mathrm{Cond}\rightleftarrows \mathrm{Shv}_k^\mathrm{fpqc} :f_*
$$
whose left adjoint $f^*$ is fully faithful and commutes with all finite limits.
\end{theorem}

\begin{proof}
The adjunction between the forgetful functor and the pearl functor gives rise to an adjunction on the opposite categories
\begin{equation}\label{Plain adj}
(-)^\circ : \CAlg_{k}^\mathrm{op}\rightleftarrows
(\CAlg_{k}^\mathrm{Stone})^\mathrm{op} :U.
\end{equation}
The pearl functor preserves $\kappa$-smallness for any cardinal $\kappa$ thanks to Proposition \ref{Pearl as a subring}, as does evidently the forgetful functor. By passing to the subcategories of $\kappa$-small rings on both sides, and then to presheaves, we obtain an adjunction
\begin{equation}\label{Contravar adj}
((-)^\circ)^*:\mP\big((\CAlg_{k, \kappa}^\mathrm{Stone})^\mathrm{op}\big)\rightleftarrows\mP(\CAlg_{k, \kappa}^\mathrm{op}) : U^*.
\end{equation}
We can compose this on the right with the sheafification adjunction of the $\kappa$-fpqc site
$$
(-)^\dagger:\mP(\CAlg_{k, \kappa}^\mathrm{op})\rightleftarrows\mathrm{Shv}^\mathrm{fpqc}_{k, \kappa},
$$
and since the functor $U^*$ preserves fpqc sheaves by Proposition \ref{site preservation lemma}, we obtain an adjunction on the level of sheaves
\begin{equation}\label{construction of geometric morphism}
(((-)^\circ)^*)^\dagger :\mathrm{Shv}\big((\CAlg_{k, \kappa}^\mathrm{Stone})^\mathrm{op}, \tau_\mathrm{fpqc}\big)
\rightleftarrows
 \mathrm{Shv}^\mathrm{fpqc}_{k, \kappa} : U^*.
\end{equation}
In fact, the left adjoint preserves finite limits, so this is a geometric morphism of topoi. Indeed, we have followed the standard procedure for obtaining a geometric morphism from a morphism of small sites, see for instance \cite[Theorem VII.10.4]{Mac Lane-Moerdijk}.

We will use Stone duality to identify the domain of the geometric morphism in question with $\kappa$-condensed sets.
If $\kappa$ is a strong limit cardinal for which $k$ is $\kappa$-small, we have the cardinality bound $|k^X| = |k|^{|X|}<\kappa$ where $X$ is any $\kappa$-small set. Since both functors of Stone duality are given as certain kinds of morphisms into $k$, this shows that  Stone duality therefore restricts to an equivalence on the subcategories of $\kappa$-small objects
$$
\mathrm{ProFin}_\kappa\simeq (\CAlg_{k,\kappa}^\mathrm{Stone})^\mathrm{op}.
$$
According to Proposition \ref{Prop1}, Stone duality is an equivalence of sites as well. We may therefore use it to identify $\kappa$-condensed sets  with sheaves on $\kappa$-small Stone $k$-algebras, or more precisely
$$
\mathrm{Cond}_\kappa\simeq \mathrm{Shv}\big((\CAlg_{k, \kappa}^\mathrm{Stone})^\mathrm{op}, \tau_\mathrm{fpqc}\big).
$$
In light of this equivalence of categories, the geometric morphism \eqref{construction of geometric morphism} may be written as
$$
f^* :\mathrm{Cond}_\kappa
\rightleftarrows
 \mathrm{Shv}^\mathrm{fpqc}_{k, \kappa} : f_*.
$$
By passing to the limit over strongly inaccessible cardinals $\kappa$ as in Proposition \ref{up the beanstalk}, this gives rise to the desired adjunction. Since the transition maps for different $\kappa$ preserve all colimits and finite limits by Proposition \ref{up the beanstalk}, the limit and colimit preservation claims concerning $f^*$ and $f_*$ follow.

To show that the functor $f^*:\mathrm{Cond}\to \mathrm{Shv}_k^\mathrm{fpqc}$ is fully faithful, it suffices to show that its restriction to the $\kappa$-small variants of the categories in question is for an arbitrary strongly inaccessible cardinal $\kappa$. That is to say, we need to show that the left adjoint of \eqref{construction of geometric morphism} is fully faithful. This is equivalent to the unit of the adjunction \eqref{construction of geometric morphism} being an isomorphism.
The unit in questions sends a functor $X:\CAlg_{k, \kappa}^\mathrm{Stone}\to \mathrm{Set}$ to the restriction to Stone $k$-algebras of the $\kappa$-fpqc sheafififcation $(X\circ (-)^\circ)^\dagger.$

We claim that the composite $X\circ(-)^\circ:\CAlg_{k, \kappa}\to \mathrm{Set}$ is a separated presheaf. To see that, let $A\to B$ be a faithfully flat map of $\kappa$-small $k$-algebras. Then $A^\circ\to B^\circ$ is also faithfully flat by Proposition \ref{site preservation lemma}. Since $X$ is a sheaf, it is in particular a separated presheaf, which means that the induced map $X(A^\circ)\to X(B^\circ)$ is injective. Thus the presheaf $X\circ (-)^\circ$ is separated.
As a separated presheaf, its sheafification is computed by the Grothendieck plus construction (in which we can further restrict to single-element covers thanks to $X\circ (-)^\circ$ commuting with finite products - indeed, the pearl functor is a right adjoint)
$$
(X\circ (-)^\circ)^\dagger(A) = \varinjlim_{B\in \mathrm{CAlg}_{A, \kappa}^\mathrm{fp}}
\operatorname{Eq}\big(X(B^\circ)\rightrightarrows X((B\otimes_A B)^\circ)\big),
$$
where $\mathrm{CAlg}_{A, \kappa}^\mathrm{fp}$ denotes the category of all $\kappa$-small faithfully flat $A$-algebras.
We only need to consider the restriction of this sheaf to Stone $k$-algebras, so we assume from now on that $A\in\CAlg_{k, \kappa}^\mathrm{Stone}$.
In that case, Lemma \ref{Pearl and products} shows that  the parallel arrows in the equalizer above factor as
$$X(B^\circ)\rightrightarrows X(B^\circ\otimes_AB^\circ)\to X((B\otimes_A B)^\circ).$$
The single arrow in this diagram is, again by Lemma \ref{Pearl and products}, induced by a faithfully flat map $B^\circ\o_A B^\circ\to (B\o_A B)^\circ$.  Since $X$ is a sheaf and therefore also a separated presheaf, it follows that the single arrow in question is injective and as such does not affect the equalizer. Therefore  we have
\begin{eqnarray*}
(X\circ (-)^\circ)^\dagger(A) &=& \varinjlim_{B\in \mathrm{CAlg}_{A, \kappa}^\mathrm{fp}}\operatorname{Eq}\big(X(B^\circ)\rightrightarrows X((B\otimes_A B)^\circ)\big)\\
 &=& \varinjlim_{B\in \mathrm{CAlg}_{A, \kappa}^\mathrm{fp}}\operatorname{Eq}\big(X(B^\circ)\rightrightarrows X(B^\circ\o_A B^\circ)\big)\\
  &=& \varinjlim_{B\in \mathrm{CAlg}_{A, \kappa}^\mathrm{fp}}X(A) \\
  &=& X(A),
\end{eqnarray*}
in which we have used that $A =A^\circ\to B^\circ$ is faithfully flat by Proposition \ref{site preservation lemma}. By chasing through the definitions, we see that the so-obtained isomorphism is precisely the unit $X\mapsto (X\circ (-)^\circ)^\dagger\circ U$ of adjunction \eqref{construction of geometric morphism}.
\end{proof}

\begin{remark}
Let $\mC$ be a category which has all small limits. Then for any site $(\mD, \tau)$, the category of $\mC$-valued sheaves on $\mD$ may be identified with the full subcategory in $\Fun(\mathrm{Shv}(\mD, \tau), \mC)$ of functors which preserve all finite limits. Using this observation, we can extend the conclusion of Theorem \ref{FF on topoi} for any uncountable strong limit cardinal $\kappa$ to an adjunction between $\kappa$-condensed objects in $\mC$ and $\mC$-valued $\kappa$-fpqc sheaves over $k$. By further passing to the filtered colimit along strongly inaccessible cardinals $\kappa$ to adjunction between the (evidently-defined; see \cite[Appendix to Lecture II]{Condensed} for the condensed case) categories
$$
f^* : \mathrm{Cond}(\mC)\rightleftarrows \mathrm{Shv}^\mathrm{fpqc}_{k}(\mC) : f_*,
$$
whose left adjoint is again fully faithful.
\end{remark}

\begin{remark}\label{sheafification formula}
By passing to the colimit over strongly inaccessible cardinals $\kappa$ from the presheaf-level adjunction for $\kappa$-variants \eqref{Contravar adj}, we obtain by \cite[Proposition A.2]{Dirac2} an adjunction on the level of accessible presheaves
$$
((-)^\circ)^*: \mP_\mathrm{acc}\big((\CAlg_k^\mathrm{Stone})^\mathrm{op}\big)\rightleftarrows \mP_\mathrm{acc}(\CAlg_k^\mathrm{op}) : U^*.
$$
It follows from the construction that this is the restriction to accessible presheaves of the adjunction that \eqref{Plain adj} induces on all presheaves. The adjunction from the statement of Theorem \ref{FF on topoi} is therefore induced from this one by passage to the subcategory of sheaves. In particular, the left adjoint $f^*:\mathrm{Cond}\to \mathrm{Shv}_k^\mathrm{fpqc}$ can be described without passing to $\kappa$-small versions as the composite
$$
\mathrm{Cond}\subseteq \mP_\mathrm{acc}(\mathrm{ProFin})\simeq \mP_\mathrm{acc}\big((\CAlg_k^\mathrm{Stone})^\mathrm{op}\big)\xrightarrow{((-)^\circ)^*}\mP_\mathrm{acc}(\CAlg_k^\mathrm{op})\xrightarrow{(-)^\dagger}\mathrm{Shv}_k^\mathrm{fpqc}.
$$
The sheafification description of the functor $f^*$ can thus also be phrased in terms of sheafification of accessible presheaves, and the compatibility of this description with the $\kappa$-version is a consequence of Remark \ref{compatibility of sheafification}.
\end{remark}

\subsection{The p\'etra space, enhanced \texorpdfstring{$k$}{k}-points, and Stone core}\label{Subsection 2.3}
\textit{Throughout this section, unless explicitly stated otherwise,  let $k$ be an arbitrary fixed field.}
Let us adopt the following more evocative notation for the functors coming from Theorem \ref{FF on topoi}.

\begin{definition}
Let $X$ be a condensed set and $\mX$ an fpqc sheaf $\mX$ over $k.$
\begin{itemize}
\item The  \textit{enhanced $k$-points on $\mX$} is the condensed set $$\underline{\mX(k)} \coloneq f_*(\mX).$$
\item The \textit{p\'etra space of $X$} is the fpqc sheaf over $k$  $$X_k^\mathrm{p\acute{e}t} \coloneq f^*(X).$$
\item The \textit{Stone core} (or \textit{petrifaction}) \textit{of $\mX$} is the fpqc sheaf over $k$ $$\mX^\mathrm{Stone} \coloneq f^*f_*(\mX).$$
\end{itemize}
\end{definition}

The goal of this subsection is to unpack the definition of these functors, give more explicit constructions of them, and try to understand them for some examples.

\subsubsection{The enhanced $k$-points of an fpqc sheaf}\label{No field subsection}

\begin{prop}\label{How enh works}
The enhanced $k$-points of an fpqc sheaf $\mX$ over $k$ is the condensed set $\underline{\mX(k)}$ whose value on any $S\in \mathrm{ProFin}$ is isomorphic to
$$
\underline{\mX(k)}(S)\simeq\mX(\CMcal C(S, k)).
$$
In particular, its set of points is $\underline{\mX(k)}(*) \simeq \mX(k)$, the usual set of $k$-points of $\mX$.
\end{prop}

\begin{proof}
In the proof of Theorem \ref{FF on topoi}, the relevant functor $f_*(\mX) =\underline{\mX(k)}$ is given by the restriction $\mX\vert_{\CAlg^\mathrm{Stone}_{k}}$. Since according to Theorem \ref{Stone theorem}, the Stone duality between profinite sets and Stone $k$-algebras sends a profinite set $S$ to the $k$-algebra  of continuous functions $\CMcal C (S, k)$, this amounts to the desired formula.
\end{proof}

\begin{exun}\label{enh of A}
The affine line $\mathbf A^1_k$ corresponds, via the functor of points perspective, to the forgetful functor $\CAlg_k\to \mathrm{Set}$. Its enhanced $k$-points $\underline{\mathbf A^1_k(k)}$ are therefore given by the profinite set $S\mapsto \CMcal C(S, k)$, which is nothing but $k$ itself, viewed as a discrete condensed set. The enhanced $k$-points functor $\mX\mapsto \underline{\mX(k)}$ is a right adjoint and as such preserves all limits. It follows that it sends the affine $n$-space $\mathbf A^n_k = \mathbf A^1_k\times\cdots \times \mathbf A^1_k$ into the constant profinite set $\underline{\mathbf A^n_k(k)} = k^n$.
\end{exun}

Though Theorem \ref{FF on topoi} only exhibits the set of $k$-points $\mX(k)$ as the underlying set of a condensed set $\underline{\mX(k)}$, it is in fact possible to extend this construction to $A$-points $\mX(A)$ for any $A\in\CAlg_{k}$.

\begin{prop}\label{Condensed enhancement}
Let $\mX$ be an fpqc sheaf over a ring $k$. It can be canonically extended to a functor $\underline{\mX} : \CAlg_{k}\to\mathrm{Cond}$ such that
\begin{itemize}
\item The $k$-points of $\underline{\mX}$ are the enhanced $k$-points of $\mX$, i.e.  $\underline{\mX}(k) = \underline{\mX(k)}$.
\item The post-composite of $\underline{\mX}$ with the underlying set functor $\mathrm{Cond}\to \mathrm{Set}$, $X\mapsto X(*)$, recovers the functor $\mX :\CAlg_{k}\to \mathrm{Cond}$. 
\end{itemize}
\end{prop}

\begin{proof}
For any fixed $\kappa$-small $k$-algebra $A$, there is a canonical adjunction
$$
(-)_A :\mathrm{Shv}_{k}^\mathrm{fpqc}\rightleftarrows \mathrm{Shv}_{A}^\mathrm{fpqc} :\mathrm{Res}_{A/k},
$$
where $\mX_A = \mX\times_{\Spec(k)}\Spec(A)$ is the base-change (equivalently, it is the restriction of $\mX:\CAlg_{k}\to \mathrm{Set}$ along the forgetful functor $\CAlg_{A}\to\CAlg_{k}$) and the Weil restriction $\mathrm{Res}_{A/k}(\mY) :\CAlg_{k}\to \mathrm{Set}$ is given by $B\mapsto \mY(A\otimes_k B)$. The comonad of this adjunction is the functor $\mX\mapsto \mathrm{Res}_{A/k}(\mX_A)$, which produces a fpqc sheaf with $k$-points
$$
\mathrm{Res}_{A/k}(\mX_A)(k) = \mX(A\otimes_k k)=\mX(A).
$$
For any $k$-algebra homomorphism $A\to B$, there is an natural map of fpqc sheaves $\mathrm{Res}_{A/k}(\mX_A)\to \mathrm{Res}_{B/k}(\mX_B)$ which induces on $k$-points the map $\mX(A)\to\mX(B)$. We can now use the enhanced $k$-points to define
$$
\underline{\mX}(A) \coloneq\underline{\mathrm{Res}_{A/k}(\mX_A)(k)},
$$
 which is functorial by the above discussion and clearly satisfies all the desiderata.
\end{proof}

\begin{remark}
It follows from the proof of Proposition \ref{Condensed enhancement} that the functor $\underline\mX$ is given explicitly as
$$
\underline \mX(A): S\mapsto \mX(\CMcal  C(S, k)\otimes_k A)
$$
for any  $k$-algebra $A$ and  small profinite set $S$. By writing $S=\varprojlim_i S_i$ for a filtered system of finite sets $S_i$, we find that
$$
\CMcal C(S, k)\otimes_k A = \varinjlim_i k^{S_i}\otimes_k A = \varinjlim_i A^{S_i} = \CMcal C(S, A).
$$
That expresses the $\kappa$-condensed enhancement $\underline \mX$ via the simpler formula
$$
\underline \mX(A) :S\mapsto \mX(\CMcal C(S, A)).
$$
It is immediate from this description that $\underline \mX$ is in fact a $\kappa$-fpqc sheaf of $\kappa$-condensed sets.
\end{remark}

\begin{remark}
The $\kappa$-condensed enhancement $\underline \mX$ of a $\kappa$-fpqc sheaf $\mX$, as furnished by Proposition \ref{Condensed enhancement} stands in contrast with the more naive way of extending $\mX:\CAlg_{k, \kappa}\to\mathrm{Set}$ to a functor into $\kappa$-condensed sets -- by post-composing with the inclusion $\mathrm{Set}\to\mathrm{Cond}_\kappa$ of discrete $\kappa$-condensed sets. 
\end{remark}

Recall that the \textit{de Rham space} of a functor $\mX : \CAlg_k\to \mathrm{Set}$ is the functor given by $\mX_\mathrm{dR}(A) \coloneq \mX(A^\mathrm{red})$, and it satisfies fpqc descent if $\mX$ does. It is the right adjoint of the reduction functor $\mX\mapsto \mX^\mathrm{red}$, extended along colimits from affines. Both functors clearly preserve $\mathrm{Shv}_{k, \kappa}^\mathrm{fpqc}$ for any choice of $\kappa$, thus their adjunction extends to $\mathrm{Shv}_k^\mathrm{fpqc}$

\begin{prop}\label{Enhanced of de Rham}
Let $\mX$ be an fpqc sheaf over $k$. The map of condensed sets
$$
\underline{\mX(k)}\to \underline{\mX_\mathrm{dR}(k)},
$$
inbduced by the canonical map of fpqc sheaves $\mX\to \mX_\mathrm{dR}$, is an isomorphism.
\end{prop}

\begin{proof}
Let $X$ be an arbitrary condensed set.
Recalling that $\underline{\mX(k)}=f_*(X)$, we find by using the adjunctions between $f^*$ and $f_*$ and between the reduction and the de Rham space a natural isomorphism
$$
\Hom_{\mathrm{Cond}}(X, \underline{\mX_\mathrm{dR}(k)})\, \simeq \, \Hom_{\mathrm{Shv}_k^\mathrm{fpqc}}(f^*(X)^\mathrm{red}, \mX).
$$
The desired claim thus reduces to showing that the canonical map $f^*(X)^\mathrm{red}\to f^*(X)$ is an isomorphism for any $X\in \mathrm{Cond}$.
The functor $f^*$ is a left adjoint and as such commutes with colimits. It follows that we may write $f^*(X)=\varinjlim_i \Spec(A_i)$ for some diagram of Stone $k$-algebras $A_i$. Since the reduction commutes with colimits as well, it suffices to show that any Stone $k$-algebra is reduced. That follows from \cite[\href{https://stacks.math.columbia.edu/tag/0CKR}{Lemma 0CKR}]{stacks-project} in light of Proposition \ref{Is wetal}.
\end{proof}

Since any infinitesimal extension $\mX\to \mY$ induces an isomorphism $\mX_\mathrm{dR}\simeq \mY_\mathrm{dR}$ on de Rham spaces, Proposition \ref{Enhanced of de Rham} asserts the insensitivity of enhanced $k$-points to nilpotent fuzz.

\subsubsection{The p\'etra space of a condensed set}

\begin{prop}\label{Stone space in terms of pi_0}
The p\'etra  $k$-space  $X^\mathrm{p\acute{e}t}_k:\CAlg_{k}\to\mathrm{Set}$  of a condensed set $X$ is the fpqc sheafification of the functor given for any $k$-algebra $A$ by 
$$
A\mapsto X(\pi_0(\left|\Spec(A)\right|)),
$$
where the set of connected components is equipped with its quotient topology. For any profinite set $S$, viewed as a condensed set, its Stone $k$-space is the affine $k$-scheme
$$
S^\mathrm{p\acute{e}t}_k \simeq \Spec (\CMcal C(S, k)).
$$
\end{prop}

\begin{proof}
The first part follows from the sheafification formula for $f^*$ from the proof of Theorem \ref{FF on topoi} in terms of the pearl functor (which is independent of the choice of $\kappa$ by Remark \ref{sheafification formula}), and Proposition \ref{Pearl as pi_0}. For the second part, we must show that
$$
A\mapsto \Spec(\CMcal C(S, k))(A^\circ) = \Hom_{\CAlg_k^\mathrm{Stone}}(\CMcal C(S, k), A^\circ)
$$
is a sheaf for the fpqc topology. By the universal property of the pearl we have
$$
 \Hom_{\CAlg_k^\mathrm{Stone}}(\CMcal C(S, k), A^\circ) = \Hom_{\CAlg_k}(\CMcal C(S, k), A) = \Spec(\CMcal C(S, k))(A),
 $$
 from which this being an fpqc sheaf is clear.
\end{proof}

\begin{prop}\label{Prop points unit}
Let $X$ be a condensed set. There is a natural isomorphism of condensed sets $X\simeq \underline{X^\mathrm{p\acute{e}t}_k(k)}$ of it with the enhanced $k$-points of its p\'etra space.
\end{prop}

\begin{proof}
The functor in question $\underline{X^\mathrm{p\acute{e}t}_k(k)} = f_*f^*(X)$ is the monad of the adjunction of Theorem \ref{FF on topoi}. Since the left adjoint $f^*$ is fully faithful, the corresponding counit map $X\to f_*f^*(X)$ is a natural isomorphism.
\end{proof}

\begin{corollary}\label{points are points}
There is a natural bijection $X(*)\simeq X_k^\mathrm{p\acute{e}t}(k)$ between the  ``underlying set'' of an arbitrary condensed set $X$ and the $k$-points of its p\'etra space.
\end{corollary}

\begin{proof}
Combine the preceding Proposition \ref{Prop points unit} with the description of enhanced $k$-points from Proposition \ref{How enh works} in the case of $S=*$.
\end{proof}

\begin{remark}
The passage from topological spaces to condensed sets adds new objects such as $Q=\mathbf R/\mathbf R^\delta$ -  the quotient of the real line with its usual topology and the real line with the discrete topology. This condensed set satisfies  $Q(*)=*$, yet $Q$ and $*$ are not isomorphic. This shows that the underlying set $X(*)$ contains a lot less information about a condensed set $X$ than it does about a topological space. Corollary \ref{points are points} formalizes the connection between this and an analogous situation in algebraic geometry. There, going from $k$-varieties to $k$-schemes likewise 
 means that the $k$-points carry much less information.
\end{remark}

The most inconvenient aspect of the description of p\'etra spaces from  Proposition \ref{Stone space in terms of pi_0} is that it involves fpqc sheafification. We now show that this sheafification, as well as passage to connected components, are sometimes unnecessary.

\begin{prop}\label{FOP for Haus} 
Let $X$ be a compactly generated Hausdorff topological space. Its p\'etra $k$-space  is given by
$$
\CAlg_k\ni A\mapsto \Hom_\mathrm{Top}(\left |\Spec(A)\right|, X)\in \mathrm{Set}.
$$
\end{prop}

\begin{proof}
In light of Proposition \ref{Stone space in terms of pi_0}, the p\'etra space is the fpqc sheafification of the presheaf
$$
A\mapsto\Hom_\mathrm{Cond}(\pi_0(\left |\Spec(A)\right |), X)\simeq \Hom_\mathrm{Top}(\pi_0(\left |\Spec(A)\right |), X).
$$
We have used that both topological spaces are compactly generated weakly Hausdorff, so that the Hom sets between them, computed in condensed sets and topological spaces respectively, agree.

Recall from \cite{Hausdorffization} that the forgetful functor $\mathrm{Haus}\to\mathrm{Top}$ admits a left adjoint, given object-wise by an iterated quotient construction. Each successive step of the construction takes a topological space $T$ to the quotient $T/\sim_\mathrm{equv}$ under the equivalence relation $\sim_\mathrm{equiv}$ generated by the (non-transitive) relation $\sim$ defined by $x\sim y$ holding if and only if $U\cap V\ne \emptyset$ for all open neighborhoods $U\ni x$ and $V\ni y$.
When this is applied to a spectral space $T=\left|\Spec(A)\right|$, if follows from \cite[\href{https://stacks.math.columbia.edu/tag/0904}{Lemma 0904}]{stacks-project} that it is precisely quotienting down to the connected components. In particular, the iterated process terminates after one step, producing the Hausdorff space\footnote{In the same vein, by
\cite[\href{https://stacks.math.columbia.edu/tag/0905}{Lemma 0905}]{stacks-project} a spectral space is Hausdorf if and only if it is profinite.} $\pi_0(T)$. 

The takeaway of the discussion in the preceding paragraph, in view of the Hausdorff assumption on $X$, is the existence of a natural bijection
$$
 \Hom_\mathrm{Top}(\pi_0(\left |\Spec(A)\right |), X) \simeq  \Hom_\mathrm{Top}(\left |\Spec(A)\right |, X).
$$
To prove the assertion of the Proposition, it suffices to show that this is an fpqc sheaf in the variable $A\in\CAlg_k$.
For this purpose, note that if $A\to B$ is a faithfully flat map, then
$$
\left |\Spec(B\o_A B)\right|\rightrightarrows\left |\Spec(B)\right | \to \left |\Spec(A)\right |
$$
is a coequalizer diagram in $\mathrm{Top}$ by \cite[Proposition 1.6.2.4]{SAG} (specifically by (2') in the proof). Applying the representable functor $\Hom_{\mathrm{Top}}(-, X)$ to this coequalizer diagram gives rise to an equalizer diagram, which exhibits the desired sheaf condition.
\end{proof}

\begin{remark}
We might wish to apply the formula for the p\'etra space of Proposition \ref{FOP for Haus} to any quasi-separated condensed set $X$, but it is not clear how to do that. The issue is that the topological space $\left |\Spec(A)\right |$ for a $k$-algebra $A$ can very well fail to be $T_1$, which means by \cite[the discussion following Proposition 2.15]{Condensed} that it does not (in the naive way) embed into condensed sets. It is therefore not clear how to make sense of a Hom set of the form $\Hom(\left |\Spec(A)\right |, X)$, i.e.\ in what category to take it.
\end{remark}

\subsubsection{The Stone core of an fpqc sheaf}

\begin{prop}\label{Stone core desc}
Let $\mX$ be a $\kappa$-fpqc sheaf over $k$ for some strongly inaccessible cardinal $\kappa$. Its Stone core can be expressed as the colimit
$$
\mX^\mathrm{Stone}=\varinjlim_{A\in (\CAlg_{k, \kappa}^\mathrm{Stone})_{/\mX}}\Spec(A)
$$
in $\mathrm{Shv}_{k}^\mathrm{fpqc}$, where $(\CAlg_{k, \kappa}^\mathrm{Stone})_{/\mX}$ denotes the category of pairs $(A, x)$ of a $\kappa$-small Stone $k$-algebra $A$ and a point $x\in \mX(A)$. There is a canonical morphism $\mX^\mathrm{Stone}\to \mX$, and $\mX$ is a p\'etra space, i.e. belongs to the essential image of the fully faithful functor $\mathrm{Cond}\to \mathrm{Shv}^\mathrm{fpqc}_{k}$ of Theorem \ref{FF on topoi}, if and only if this  is an isomorphism.
\end{prop}

\begin{proof}
The Stone core is the comonad $f^*f_*$ of the adjuncton constituting the geometric morphism $f: \mathrm{Cond}_\kappa\to \mathrm{Shv}^\mathrm{fpqc}_{k, \kappa}$. We saw in the proof of Theorem \ref{FF on topoi} that $f_*(\mX)$ can be identified, as a functor $\CAlg_{k, \kappa}^\mathrm{Stone}\to \mathrm{Set}$, with the restriction $\mX\vert_{\CAlg_{k, \kappa}^\mathrm{Stone}}$. Any presheaf $X :\CAlg_{k, \kappa}^\mathrm{Stone}\to\mathrm{Set}$ is equal to the colimit in the functor category of representables, indexed by maps $\Spec(A)\to X$ with $A\in \CAlg_{k, \kappa}^\mathrm{Stone}$. For $X=\mX\vert_{\CAlg_{k, \kappa}}^\mathrm{Stone}$, this indexing category is precisely $(\CAlg_{k, \kappa}^\mathrm{Stone})_{/\mX}$. The colimit formula for the Stone core now follows from noting that sheafification $(-)^\dagger :\Fun(\CAlg_{k, \kappa}, \mathrm{Set})\to \mathrm{Shv}^\mathrm{fpqc}_{k, \kappa}$, as a left adjoint, preserves colimits. The remaining statement follows from noting that we have $f_*f^*=\id$ by the proof of Theorem \ref{FF on topoi}, so that we have
$$
(\mX^\mathrm{Stone})^\mathrm{Stone} = f^*f_*f^*f_*(\mX) = f^*f_*(\mX) = \mX^\mathrm{Stone},
$$
while also for any $\kappa$-condensed set $X$  we get
$$
(X^\mathrm{p\acute{e}t}_k)^\mathrm{Stone} = f^*f_*f^*(X) = f^*(X) = X^\mathrm{p\acute{e}t}_k.
$$
That is to say, the Stone core is an idempotent monad, and its fixed points is the essential image of the p\'etra space functor.
\end{proof}

\begin{remark}\label{Stone points}
Though the Stone core $\mX^\mathrm{Stone}$ of a fpqc sheaf $\mX$ is in general difficult to describe more explicitly than in Proposition \ref{Stone core desc}, its $k$-points are very simple. Indeed, by combining Proposition \ref{How enh works} and the fact that $f^*f_* \simeq \mathrm{id}$ in the proof of Theorem \ref{FF on topoi}, we find that
$$
\mX^\mathrm{Stone}(k) = \underline{\mX^\mathrm{Stone}(k)}(*) = f_*f^*f_*(\mX)(*) \simeq f_*(\mX)(*) = \underline{\mX(k)}(*) = \mX(k).
$$
Unwinding the definitions, we find that the Stone core map $\mX^\mathrm{Stone}\to \mX$ induces a bijection on $k$-points, and more generally on $A$-points for all Stone $k$-algebras $A$.
\end{remark}

\begin{prop}\label{Stone core via pearl}
The Stone core of an fpqc sheaf $\mX$ over $k$ is the fpqc sheafification of the functor $A\mapsto X(A^\circ)$.
\end{prop}

\begin{proof}
We saw this description in the proof of Theorem \ref{FF on topoi}.
\end{proof}

\begin{remark}
The description of the Stone core in Proposition \ref{Stone core via pearl} is somewhat reminiscent of the formation of the de Rham space - see the discussion before Proposition \ref{Enhanced of de Rham}. The latter is defined for any fpqc space $\mX$ as $$
\mX_\mathrm{dR}(A) \coloneq \mX(A^\mathrm{red}),
$$
from which we can glean two important differences. Firstly, no sheafification is required in the de Rham space. Secondly, the pearl of a $k$-algebra comes with a natural map $A^\circ\to A$, while the reduction comes with a map $A\to A^\mathrm{red}$. Thus while the de Rham space is a quotient $\mX\to \mX_\mathrm{dR}$, the Stone core is a subspace $\mX^\mathrm{Stone}\to \mX$.
\end{remark}

The de Rham space, other than as an analogy to the Stone core, can also be used to express the latter's insensitivity to infinitesimal extensions.

\begin{prop}
For any fpqc sheaf $\mX$ over $k$, there are canonical isomorphisms
$$
\mX^\mathrm{Stone}\simeq 
(\mX_\mathrm{dR})^\mathrm{Stone}\simeq (\mX^\mathrm{Stone})_\mathrm{dR}.
$$
\end{prop}

\begin{proof}
Since the Stone core is defined as $\mX^\mathrm{Stone}= \underline{\mX(k)}^\mathrm{p\acute{e}t}_k$, this follows from the
invariance of the enhanced $k$-points under passage to the de Rham space of Proposition \ref{Enhanced of de Rham}.
\end{proof}

In some cases, e.g.\ by Corollary \ref{left adjoint Fp} when $k=\mathbf F_p$ is a prime field, the forgetful functor from Stone $k$-algebras to $k$-algebras admits a left adjoint. In that case, we can describe the Stone core of affines, and consequently enhanced $k$-points, quite explicitly.

\begin{prop}\label{Q implies}
Let $k$ be a field for which the forgetful functor $U: \CAlg_{k}^\mathrm{Stone}\to\CAlg_{k}$ admits a left adjoint $Q : \CAlg_{k} \to \CAlg_{k}^\mathrm{Stone}$. Then the Stone core of an affine $k$-scheme is given by $\Spec(A)^\mathrm{Stone} = \Spec(Q(A))$.
\end{prop}

\begin{proof}
Let $A$ be a $k$-algebra, and choose a strongly inaccessible cardinal $\kappa$ for which both $A$ and $Q(A)$ are $\kappa$-small.
Then the Stone $k$-algebra $Q(A)$ is for any $\kappa$-small $k$-algebra $A$ initial object in the category
$$
(\CAlg_{k, \kappa}^\mathrm{Stone})_{/\Spec(A)}\simeq (\CAlg_{k, \kappa}^\mathrm{Stone})_{A/}.
$$
The claim now follows directly from the formula for the Stone core of Proposition \ref{Stone core desc}.
\end{proof}

\begin{exun}\label{ex with Q}
Let $k$ be as in Proposition \ref{Q implies} and $A$ any $k$-algebra. The enhanced $k$-points of its spectrum are given by
$$
\underline{\Spec(A)(k)} =\underline{\Spec(A)^\mathrm{Stone}(k)} = \underline{\Spec(Q(A))(k)}= \left|\Spec(Q(A))\right|,
$$
and in particular form a profinite set.
\end{exun}

\begin{exun}
Let us continue within the setting of Example \ref{ex with Q}, with the additional assumption that the $k$-algebra $A$ is of finite type. That means that there is a closed immersion, and therefore a monomorphism, into some affine space $\Spec(A)\to \mathbf A^n_k$. The enhanced $k$-points functor is a right adjoint and as such preserves monomorphisms. It follows from Example \ref{enh of A} that $\underline{\Spec(A)(k)}\to \underline{\mathbf A^n_k(k)}=k^n$ is a  condensed subset of the discrete set. Since we already know it is profinite as well, it follows that it is a profinite subset of a discrete set, and hence itself discrete. It therefore equals its underlying set
$$
\underline{\Spec(A)(k)} = \Spec(A)(k),
$$
the ordinary set of $k$-points viewed as a discrete condensed set.
\end{exun}

The preceding example suggests that the enhanced $k$-points might be of most interest for schemes of infinite type, at least over a base field such as $k=\mathbf F_p$.

\begin{remark}
Let $k=\mathbf F_p$ for a prime number $p$. According to the proof of corollary \ref{left adjoint Fp}, the left adjoint $Q:\CAlg_{\mathbf F_p}\to \CAlg_{\mathbf F_p}^\mathrm{Stone}$ is given by $Q(A) = A/(a-a^p: a\in A)$. That it to say, $\Spec(Q(A))$ is identified with the fixed-point locus $\Spec(A)^{F}$ of the Frobenius endomorphism $F:\Spec(A)\to\Spec(A)$.
\end{remark}

\section{Stone duality: condensed sets to algebraic geometry}\label{Stone duality section}
We suggest interpreting  Theorem \ref{FF on topoi} as an extension of Stone duality of Theorem \ref{Stone theorem}, going between condensed mathematics on one hand and algebraic geometry on the other. 
The key aspect in which Theorem \ref{FF on topoi} is to be viewed as a duality is that many familiar properties and operations on one side correspond to ones on the other side, but which are often not \textit{a priori} obviously encoding the same thing. This section sets out to provide a repertoire of such examples.

\begin{remark}
One possible objection to this interpretation is that the correspondence of Theorem \ref{FF on topoi} is that it does not give an \textit{equivalence} between both sides, as a duality likely should. But if we restrict on the algebro-geometric side to only those fpqc sheaves which equal their own Stone core (i.e.\ to p\'etra spaces), that does become true as well by Proposition \ref{Stone core desc}. This is effectively what we do, but since the category of all fpqc sheaves is much more familiar, we prefer to phrase everything in terms of it instead.
\end{remark}

In Subsection \ref{Subsection 3.1} we study what affinization corresponds under Stone duality. Then we turn in Subsection \ref{Subsection 3.2} to showing that properties of a condensed set translate into algebro-geometric properties of its p\'etra space. 
In Subsection \ref{Subsection 3.3} we prove that Stone duality interchanges closed embeddings of condensed sets with closed immersions of p\'etra spaces. We wish to extend this to open embeddings and immersions respectively, and we do so in Subsection \ref{Subsection 3.5}, after first introducing and studying the notion of an open embedding of condensed sets in Subsection \ref{Subsection 3.4}.
 In Subsection \ref{Subsection 3.6} we discuss sheaves of modules in both the condensed and algebro-geometric setting, and show that they agree under Stone duality. From this, we deduce a Stone reconstruction result for compact Hausdorff spaces in Subsection \ref{Subsection 3.7}.\\

\textit{Throughout this section, unless explicitly stated otherwise, let $k$ be an arbitrary  field.}\\

\subsection{Global functions and affinization}\label{Subsection 3.1}
Both the condensed side and algebro-geometric side possess a natural notion of ``global functions'' - condensed maps into the fixed field $k$ on the one hand, and global sections of the structure sheaf on the other. As a first instance of Stone duality, we show that these coincide, and in fact not just over a field.

\begin{prop}\label{Global functions}
Let $X$ be a condensed set and let $\mX=X^\mathrm{p\acute{e}t}_k$ be its p\'etra space.
There is a canonical isomorphism of $k$-algebras
$$
\Gamma(\mX; \CMcal O_{\mX})\simeq \Hom_{\mathrm{Cond}}(X, k).
$$
\end{prop}

\begin{proof}
We have a chain of $k$-algebra isomorphisms
\begin{eqnarray*}
\Gamma(\mX; \CMcal O_{\mX}) &\simeq&\Hom_{\mathrm{Shv}_k^\mathrm{fpqc}}(\mX, \mathbf A^1_k) \\
&=&\Hom_{\mathrm{Shv}_k^\mathrm{fpqc}}(f^*(X), \mathbf A^1_k) \\
&\simeq&\Hom_{\mathrm{Cond}}(X, f_*(\mathbf A^1_k)) \\
&=&\Hom_{\mathrm{Cond}}(X, \underline{\mathbf A^1_k(k)})) \\
&\simeq &\Hom_{\mathrm{Cond}}(X, k),
\end{eqnarray*}
where the final isomorphism comes from Example \ref{enh of A}.
\end{proof}

We can leverage Proposition \ref{Global functions}  and Stone duality to prove the following result about condensed sets, which does not seem immediate otherwise - see Remark \ref{Well-defined non-obvious}.

\begin{prop}\label{left adjoint profinite components}
The inclusion functor $\mathrm{ProFin}\to \mathrm{Cond}$, obtained by passage to colimit over uncountable strong limit cardinals $\kappa$ from the Yoneda inclusions $\mathrm{ProFin}_\kappa\to\mathrm{Cond}_\kappa$,
 admits a left adjoint $\widehat\pi_0:\mathrm{Cond}\to\mathrm{ProFin}$.
\end{prop}

\begin{proof}
Let $X$ be a condensed set, $S$ a profinite set, and $k$ a \textit{finite} field. By Theorem \ref{FF on topoi} we have the first in the chain of natural bijections
\begin{eqnarray*}
\Hom_{\mathrm{Cond}}(X, S)
&\simeq& \Hom_{\mathrm{Shv}_k^\mathrm{fpqc}}\big(X^\mathrm{p\acute{e}t}_k, S^\mathrm{p\acute{e}t}_k\big)\\
&\simeq&
 \Hom_{\mathrm{Shv}_k^\mathrm{fpqc}}\big(X^\mathrm{p\acute{e}t}_k, \Spec(\CMcal C(S, k))\big)\\
 &\simeq&
 \Hom_{\CAlg_k}\big(\CMcal C(S, k), \Hom_{\mathrm{Cond}}(X, k)\big),
\end{eqnarray*}
 with the second coming from Proposition \ref{Stone space in terms of pi_0} and the third from Proposition \ref{Global functions} and the usual adjunction between the spectrum and global functions in algebraic geometry.
 Now choose a strongly inaccessible cardinal $\kappa$ such that $X\in \mathrm{Cond}_\kappa$.  The finite field $k$ evidently also belongs to $\mathrm{Cond}_\kappa$, and since the left Kan extension $\mathrm{Cond}_\kappa\to\mathrm{Cond}$ is fully faithful, we can obtain natural isomorphisms
$$
\Hom_{\mathrm{Cond}}(X, k)
\simeq 
\Hom_{\mathrm{Cond}_\kappa}(X, k)
\simeq \Hom_{\mathrm{Top}}(X(*)_\mathrm{top}, k).
$$
Here $X(*)_\mathrm{top}$ denotes the underlying topological space of $X$, i.e.\ its image under the left adjoint to the canonical functor $\mathrm{Top}\to \mathrm{Cond}_\kappa$ of \cite[Proposition 1.7]{Condensed}. The Stone-\v{C}ech compactification $\beta: \mathrm{Top}\to\mathrm{CHaus}$ is the left adjoint to the forgetful functor, and so due to $k$ being assumed finite and therefore compact Hausdorff, we have
$$
\Hom_{\mathrm{Top}}(X(*)_\mathrm{top}, k)\simeq \Hom_{\mathrm{CHaus}}(\beta(X(*)_\mathrm{top}), k)
$$
On the other hand, $k$ is discrete, and so any continuous map from another set into it must contract all connected components. As in  Proposition \ref{Pearl as pi_0}, the set of connected components $\pi_0(X)$ is totally disconnected with respect to the connected topology by \cite[\href{https://stacks.math.columbia.edu/tag/08ZL}{Lemma 08ZL}]{stacks-project} and finite due to $X$ being compact. It follows that
$$
 \Hom_{\mathrm{CHaus}}(\beta(X(*)_\mathrm{top}), k) \simeq  \Hom_{\mathrm{ProFin}}(\pi_0(\beta(X(*)_\mathrm{top})), k).
$$
Let us denote $\widehat\pi_0(X) \coloneq\pi_0(\beta(X(*)_\mathrm{top}))$. By combining the above isomorphisms, we obtain
$$
\Hom_{\mathrm{Cond}}(X, S)
\simeq
 \Hom_{\CAlg_k}\big(\CMcal C(S, k), \CMcal C(\widehat\pi_0(X), k)\big)\simeq \Hom_\mathrm{ProFin}(\widehat\pi_0(X), S)
$$
 where the second bijection comes from Stone duality in the form of Theorem \ref{Stone theorem}.
 
 The functor $\widehat\pi_0$ as defined above seems to depend on the choice of the cardinal $\kappa$, and to finish the proof, we must show that it does not. But we saw in the above proof that continuous $k$-valued functions on this profinite set are the same as those on the condensed set $X$. The explicit form of  Stone duality  from Theorem \ref{Stone theorem}  furnishes a canonical homeomorphism
 \begin{equation}\label{Stone-Cech formula}
\widehat\pi_0(X)\simeq \left |\Spec(\Hom_{\mathrm{Cond}}(X, k))  \right |
\end{equation}
 which is independent of $\kappa$ and shows that the construction $X\mapsto \widehat\pi_0(X)$ is functorial with respect to an arbitrary condensed set $X$. By what we have already shown above, it is the desired left adjoint.
\end{proof}

\begin{remark}\label{Well-defined non-obvious}
The assertion of Proposition \ref{left adjoint profinite components} is non-obvious due to the fact that the functor $\mathrm{Cond}_\kappa\to\mathrm{ProFin}$, given by $X\mapsto \pi_0(\beta(X(*)_\mathrm{top}))$ seems at first glance to depend on the choice of the cardinal $\kappa$. Indeed, recall from \cite[Warning 2.14]{Condensed} that there is not in general a well-defined functor $\mathrm{Top}\to\mathrm{Cond}$ without the restriction to $\kappa$-condensed sets on the right. The topological space $X(*)_\mathrm{top}$ might therefore depend on the choice of $\kappa$ as well.  But as the proof of Proposition \ref{left adjoint profinite components} shows, that is at least not the case on connected components with the quotient topology.
\end{remark}

\begin{remark}
Let $S$ be a set. We may identify it with a discrete condensed set, and extract from it a profinite set $\widehat{\pi}_0(S)$. Since $S(*)_\mathrm{top}=S$, we find from the explicit construction in the proof of Proposition \ref{left adjoint profinite components} that $\widehat{\pi}_0(S) =\beta S$ is the Stone-\v{C}ech compactification of $S$. The algebro-geometric formula \eqref{Stone-Cech formula} specializes for $k=\mathbf F_2$ to the standard expression  for the Stone-\v{C}ech compactification $\beta S\simeq \left |\Spec(P(S))\right |$, where $P(S) = \Hom_\mathrm{Set}(S, \{0, 1\})$ is the power set of $S$.
\end{remark}

By unpacking its definition, we can show that the functor $\widehat\pi_0$ corresponds under Stone duality in the sense of Theorem \ref{FF on topoi}  to the affinization of fpqc stacks.

\begin{prop}\label{Affinization is pi_0}
Let $X$ be a condensed set and $k$ a finite field.
The affinization of its p\'etra space $X^\mathrm{p\acute{e}t}_k$ is the p\'etra space of the profinite set $\widehat\pi_0( X)$.
\end{prop}

\begin{proof}
The affinization of any fpqc sheaf $\mX$ is given by $\mathrm{Aff}(\mX) = \Spec(\Gamma(\mX; \CMcal O_{\mX}))$. By Proposition \ref{Global functions} we find the affinization of a p\'etra space given by
$$
\mathrm{Aff}(X^\mathrm{p\acute{e}t}_k)\simeq \Spec(\Hom_{\mathrm{Cond}}(X, k)).
$$
Bu construction of the profinite set $\widehat\pi_0(X)$ in the proof of Proposition \ref{left adjoint profinite components}, its continuous functions come with a canonical bijection $\CMcal C (\widehat\pi_0(X), k)\simeq \Hom_{\mathrm{Cond}}(X, k)$. The desired result now follows by the formula for the p\'etra spaces of profinite sets of Proposition \ref{Stone space in terms of pi_0}.
\end{proof}

\begin{remark}\label{Affinization for infinite fields}
Without the finiteness assumption on the field $k$, the  affinization of the p\'etra space of a $\kappa$-condensed set $X$ could still be expressed as
$$
\mathrm{Aff}(X^\mathrm{p\acute{e}t}_k)\simeq \Spec(\CMcal C(\pi_0(X(*)_\mathrm{top}), k)).
$$
The set of connected components $\pi_0(X(*)_\mathrm{top})$, equipped with the quotient topology, is  a totally disconnected space. But without assuming that the topological space $X(*)_\mathrm{top}$ is compact Hausdorff (or for instance spectral), it might not be a profinite set.
\end{remark}

\subsection{Geometric properties of p\'etra spaces}\label{Subsection 3.2}
Despite appearing somewhat exotic at first glance, we will show that p\'etra spaces of condensed sets
are in fact not very far from the types of spaces usually considered in algebraic geometry. For this, we require the following slight enlargement of the class of algebraic stacks.

\begin{definition}\label{Def of geom space}
An fpqc sheaf $\mX$ over $k$ is a \textit{geometric space} if it is isomorphic to a quotient of an affine scheme by an affine equivalence relation in the category $\mathrm{Shv}_k^\mathrm{fpqc}$.

An  fpqc sheaf $\mX$ over $k$ is an \textit{ind-geometric space over $k$}  if it is isomorphic to a filtered colimit $\mX  \simeq \varinjlim_i \mX_i$ in $\mathrm{Shv}_k^\mathrm{fpqc}$ of geometric spaces $\mX_i$ along closed immersions $\mX_i\to \mX_j$.
\end{definition}

\begin{remark}\label{Remark vs geometric stacks}
Equivalently, an accessible functor $\mX: \CAlg_k\to \mathrm{Set}$ is a geometric space if and only if it satisfies the following conditions:
\begin{enumerate}[label = (\textit{\alph*})]
\item $\mX$ satisfies fpqc descent.
\item The diagonal map $\Delta : \mX\to \mX\times \mX$ is affine.
\item There exists a $k$-algebra $A$ and a faithfully flat morphism $\Spec(A)\to \mX$.\label{Remark pt3}
\end{enumerate}
For an algebraic stack, we would only\footnote{We view the questions of whether we demand a cover by a single affine or merely a scheme, and the precise condition imposed on the diagonal, as finer points of convention and non-essential.} modify \ref{Remark pt3} to ask for an \'etale cover (or an fppf one, as turns out to be equivalent using \cite[Artin slice theorem 10.1]{Laumon-Moret-Bailly}). So while an algebraic space has an \'etale atlas, a geometric space is only guaranteed an fpqc one. Geometric spaces could therefore be called ``fpqc-algebraic space''. We have chosen instead the more evocative adjective ``geometric'', which also accords with the terminology of ``geometric stacks'' of \cite[Definition 9.3.0.1]{SAG}. Indeed, it follows from \cite[Proposition 9.1.6.9]{SAG} that our geometric spaces are precisely the $0$-truncated geometric stacks in the sense of \cite[Definition 9.1.6.2]{SAG}.
\end{remark}

\begin{remark}
The definition of ind-geometric spaces is analogous to the usual definition of ind-schemes. In particular, it is an fpqc-version of an {ind-algebraic space} in the sense of \cite[Definition 41.1]{Algebraic Magnetism}.
\end{remark}

\begin{theorem}\label{condensed is geometric}
Let $X$ be a condensed set and $\mX=X^\mathrm{p\acute{e}t}_k$ its p\'etra space over a field $k$.
\begin{enumerate}
\item If $X$ is a profinite set, then $\mX$ is affine.\label{Stone1}
\item If $X$ is a compact Hausdorff space (i.e.\ a qcqs condensed set), then $\mX$ is a geometric space.\label{Stone2}
\end{enumerate}
When $k$ is a finite field, the converse statements also hold, as well as
\begin{enumerate}
\setcounter{enumi}2
\item A condensed set $X$ is  quasi-separated if and only if $\mX$ is an ind-geometric space.\label{Stone3}
\end{enumerate}
\end{theorem}

\begin{proof}
Assertion \eqref{Stone1} follows from Proposition \ref{Affinization for infinite fields}. When $k$ is a finite field, we have by Proposition \ref{Affinization is pi_0} that
$$
\mX\simeq \mathrm{Aff}(\mX) \simeq \widehat\pi_0(X)^\mathrm{p\acute{e}t}_k,
$$
so if $\mX$ is affine, it is isomorphic to the p\'etra space of the profinite set $\widehat\pi_0(X)$. Since the p\'etra space functor is fully faithful by Theorem \ref{FF on topoi}, this is the case precisely when $X\simeq \widehat\pi_0(X)$, i.e.\ when $X$ is profinite.

For assertion \eqref{Stone2}, recall from \cite[Theorem 2.16]{Condensed} that any compact Hausdorff space $X$ may be obtained as the quotient in condensed sets $X= S/ R$ of a profinite set $S$ under a profinite equivalence relation $R\subseteq S\times S$. Conversely, since qcqs condensed sets contain profinite sets and are closed under quotients, all such quotients $S/R$ are compact Hausdorff spaces. The formation of p\'etra spaces from condensed sets preserves both finite limits and colimits, hence equivalence relations and quotients by them. Assertion \eqref{Stone2} thus follows from  \eqref{Stone1} by the definition of geometric spaces. The same is true for the converse assertion when $k$ is finite.

The argument for assertion \eqref{Stone3} uses the identification of \cite[Proposition 2.9]{Complex} between quasi-separeted condensed sets and the subcategory of $\mathrm{Ind}(\mathrm{CHaus})$ spanned by filtered colimits of compact Hausdorff spaces $\varinjlim_i X_i$ along injective transition maps $X_i\to X_j$. Assertion \eqref{Stone3} now follows from the next Lemma.
\end{proof}

\begin{lemma}\label{Lemma on closed immersions}
A continuous map between compact Hausdorff spaces $X\to Y$ is injective if and only if the morphism between the corresponding p\'etra spaces $\mX\to \mY$ over a finite field $k$ is a closed immersion.
\end{lemma}

\begin{proof}
Since the property being a closed immersion is local on the base, and $\mY$ admits by point \eqref{Stone2} of Theorem \ref{condensed is geometric} an fpqc cover by an affine, it suffices to assume that $\mY=\Spec(A)$, and so by the converse of point \eqref{Stone1}  of Theorem \ref{condensed is geometric} (applicable because $k$ is a finite field) that $Y$ is profinite. If $\mX\to \Spec(A)$ is a closed immersion, then $\mX=\Spec(B)$ is affine so, again by the converse of assertion \eqref{Stone1}  of Theorem \ref{condensed is geometric}, the corresponding condensed set $X$ is profinite. Stone duality is an anti-equivalence of categories, so it sends the epimorphism $A\to B$ in $\CAlg_k^\mathrm{Stone}$ to a monomorphism in $\mathrm{ProFin}$, i.e.\ an injective map.

Conversely, let $X\to Y$ be an injective continuous map into a profinite set. Since any continuous injection between compact Hausdorff spaces is a closed embedding, we may identify $X$ with a closed subspace of $Y$. Since $Y\simeq\left |\Spec(A)\right |$ for a Stone $k$-algebra $A$, the closed subspace $X$ corresponds to a Zariski closed subset. That is to say, to a reduced closed subscheme, i.e.\ to a surjective $k$-algebra map $A\to B$ into a reduced $k$-algebra $B$. Stone $k$-algebras are closed under quotients, so $B$ is a Stone $k$-algebra, and hence $Y\simeq \left |\Spec(B)\right |$ is a profinite space. It follows from Proposition \ref{Stone space in terms of pi_0} that the map between p\'etra space induced by $X\to Y$ is then $\Spec(B)\to \Spec(A)$, which is indeed a closed immersion.
\end{proof}

\begin{remark}
When viewed as condensed sets, all compactly generated weakly Hausdorff topological spaces are quasi-separated by \cite[Theorem 2.16]{Condensed}. Theorem \ref{condensed is geometric} therefore implies that their p\'etra spaces over a finite field are always ind-geometric.
\end{remark}

\begin{corollary}
Let $k$ be a finite field.
The p\'etra space over $k$ of any condensed space is isomorphic to a quotient, formed in the category $\mathrm{Shv}_k^\mathrm{fpqc},$ of an ind-geometric space by an ind-geometric equivalence relation.\label{Stoneii}
\end{corollary}

\begin{proof}
This follows from part \eqref{Stone3} of Theorem \ref{condensed is geometric} and the observation of \cite[discussion above Definition 2.10]{Complex} that, since disjoint unions of profinite sets are quasi-separated and any subobject of a quasi-separated condensed set is still quasi-separated, we may express any condensed set as a quotient of a quasi-separated consensed set by a quasi-separated equivalence relation. 
\end{proof}

\begin{remark}
We have chosen our terminology to accord with \cite{SAG}, but there is another common convention on the meaning of ``geometric" spaces or stacks from \cite{Simpson}, \cite{TV}, and \cite{German}. In that language, the geometric spaces in the sense of Definition \ref{Def of geom space} would be called quasi-compact geometric spaces with affine diagonal, while the ind-geometric spaces of Definition \ref{Def of geom space} would still fall inside the domain of geometric spaces. In particular, an ind-geometric space is still geometric in the sense that it admits a cover by (possibly infinitely many) affines, or equivalently of a (possibly non-affine) scheme.
\end{remark}

Though Theorem \ref{condensed is geometric} shows that p\'etra spaces are reasonable algebro-geometric objects, they are seldom actual nice algebraic spaces. In fact, the characterization of when a p\'etra space is an algebraic space locally of finite presentation may be viewed as an instance of Stone duality. Here a condensed set is said to be \textit{discrete} if it belongs to the essential image of the fully faithful left adjoint to the ``underlying set'' functor $\mathrm{Cond}\to\mathrm{Set}$, $X\mapsto X(*)$.

\begin{prop}\label{non-algebraic}
Let $X$ be a condensed set. Its p\'etra space $\mX$ over a field $k$ is an algebraic space locally of finite presentation over $\Spec(k)$ if and only if $X$ is discrete.
\end{prop}

\begin{proof}
Since Stone algebras are always weakly \'etale, it follows by   \cite[\href{https://stacks.math.columbia.edu/tag/08R2}{Lemma 08R2}]{stacks-project} that $L_{\mX/k} = 0$, i.e.\ its cotangent complex vanishes. By
\cite[\href{https://stacks.math.columbia.edu/tag/04G9}{Lemma 04G9}]{stacks-project}
this implies $\mX\to\Spec(k)$ is formally unramified, and since it is weakly \'etale it is in particular also flat, hence it is formally \'etale by \cite[\href{https://stacks.math.columbia.edu/tag/0615}{Lemma 0615}]{stacks-project}
 It is by assumption also locally of finite presentation, thus it is \'etale by \cite[\href{https://stacks.math.columbia.edu/tag/0616}{Lemma 0616}]{stacks-project} and so 
  by \cite[\href{https://stacks.math.columbia.edu/tag/03KX}{Lemma 03KX}]{stacks-project} actually a scheme. Any \'etale scheme $\mX\to\Spec(k)$ is by \cite[\href{https://stacks.math.columbia.edu/tag/02GL}{Lemma 02GL}]{stacks-project} of the form
  $$
  \mX\simeq \coprod_{i\in I}\Spec(L_i)
  $$
  for separable field extensions $L_i/k$. But since these $L_i$ must also be Stone $k$-algebras, so (e.g.\ as seen in Example \ref{Example extension}) the only possibility is that $L_i=k$ for all $i$. That means that
  $$
  \mX\simeq \coprod_{i\in I} \Spec(k)
  $$
  and due to p\'etra space functor commuting with colimits, this is the p\'etra space of the discrete condensed set $I$. Conversely, the latter is obviously an algebraic space and locally of finite presentation, showing those to be the only algebraic p\'etra spaces.
\end{proof}

\begin{remark}
The conclusions of Proposition \ref{non-algebraic} and point \eqref{Stone1} of Theorem \ref{condensed is geometric} might seem at odds. On the condensed side, they respectively concern discrete and profinite sets, while on the algebro-geometric side, they concern affine schemes and algebraic spaces locally of finite presentation. The common intersection between these two cases on the condensed side is finite sets, while on the algebro-geometric side it is affine schemes of finite presentation. It might seem like the latter class is much broader - but not if we remember that we must only consider those finitely-presented affine schemes that arise from a Stone $k$-algebra. For those, a finite presentation assumption is highly restrictive, and in fact reduces to the case where $A=k^S$ for some finite set $S$. That indeed matches the finite sets appearing on the condensed side, and so there is no discrepancy. 
\end{remark}

\begin{remark}
We have seen in the proof of Proposition \ref{non-algebraic} that $L_{\mX/k}$ holds for any p\'etra space. It follows from the fundamental cofiber sequence for cotangent complexes that any map of p\'etra spaces $\mX\to \mY$ satisfies $L_{\mX/\mY}=0$, and is therefore formally \'etale in this sense.
\end{remark}

\subsection{Closed embeddings and closed immersions}\label{Subsection 3.3}We next turn to the question of how Stone duality interacts with embeddings.

\begin{definition}[{\cite[Appendix to Lecture IV]{Analytic}}]
A \textit{closed embedding of condensed sets} $Z\to X$ is 
 a quasi-compact injection of condensed sets.
\end{definition}

 It is shown in \cite[Proposition 4.13]{Analytic} that closed embeddings are fully determined by closed subsets of the underlying topological space $Z(*)_\mathrm{top}\subseteq X(*)_\mathrm{top}$. In particular, the notion of a closed embedding recovers its usual topological meaning for compactly generated weak Hausdorff spaces.

\begin{lemma}\label{Lemma on closed condensed}
A map of condensed sets $Z\to X$ is a closed embedding if and only if it satisfies the following property: for any map from a profinite set $S\to X$, the pullback $S\times_X Z\to S$ is a closed embedding of profinite sets.
\end{lemma}

\begin{proof}
Given a closed embedding $Z\to X$, quasi-compactness implies that $S\times_X Z$ is a profinite set for any map $S\to X$ with $S\in \mathrm{ProFin}$. Injectivity of $Z\to X$ implies that $Z\times_X S\to S$ is injective as well. But since any continuous map between compact Hausdorff spaces is automatically proper,
a map of profinite sets is injective if and only if it is a closed embedding. Closed embeddings $Z\to X$ therefore satisfy the desired property.

Conversely, assume that a map of condensed sets $Z\to X$ satisfies the property that $Z\times_X S\to S$ is a closed embedding for any map from a profinite set $S\to X$. By taking the profinite set $S=*$, since its closed subsets are only the empty set and itself, we find that the map $Z\to X$ is injective. As we already noted above, $Z\times_X S$ being a profinite set also implies (and is equivalent to) quasi-compactness of $Z\to X$. Thus such a map of condensed sets $Z\to X$ is a closed embedding as desired.
\end{proof}

On the algebro-geometric side, the corresponding notion of a \textit{closed immersion} can be defined via descent from affines, thanks to the notion of a closed immersion being fpqc-local by \cite[\href{https://stacks.math.columbia.edu/tag/02L6}{Lemma 02L6}]{stacks-project}. Closed immersions, e.g.\ of schemes, can in general be much richer than mere closed subsets of underlying topological spaces, remembering infinitesimal information encoded as nilpotent fuzz. We show that for p\'etra spaces, this is not the case.

\begin{prop}\label{Closed prop}
Let  $\mX$ be the p\'etra space of a quasi-separated condensed set
$X$ over a finite field $k$. Stone duality of Theorem \ref{FF on topoi} induces a bijection between closed embeddings of condensed sets $Z\to X$ and closed immersions of fpqc sheaves $\mZ\to \mX$.
\end{prop}

\begin{proof}
Since $X$ is quasi-separated, its p\'etra space $\mX$ admits an fpqc cover by spectra of Stone $k$-algebras.
The property being a closed immersion is fpqc-local on the base, and so a map of fpqc sheaves $\mZ\to \mX$ is a closed immersion if and only if its pullback 
$\Spec(A)\times_{\mX}\mZ$ along any map $\Spec(A) \to\mX$ for $A\in\CAlg_k^\mathrm{Stone}$
is a  closed immersion into $\Spec(A)$. By combining Lemma \ref{Lemma on closed immersions} and Lemma \ref{Lemma on closed condensed}, we see that the claim of the Proposition follows if we know that $\mZ$ is itself a p\'etra space of some condensed set.

Let therefore $\mZ\to \mX$ be a closed immersion. Choose a strongly inaccessible cardinal
 $\kappa$ for which $\mX, \mZ\in \mathrm{Shv}_{k, \kappa}^\mathrm{fpqc}$. Then we can write $\mX = \varinjlim_i\Spec(A_i)$ for a small diagram of $\kappa$-small Stone $k$-algebras $A_i$. Each pullback
 $$
 \Spec(A_i)\times_{\mX}\mZ\to \Spec(A_i)
 $$
 is a closed immersion, so it is given by
 $$
 \Spec(A_i)\times_{\mX}\mZ\simeq \Spec(B_i),
 $$ where $A_i\to B_i$ is a $k$-algebra surjection. Since the class of $\kappa$-small Stone $k$-algebras is closed under quotients, we have $B_i\in\CAlg_{k, \kappa}^\mathrm{Stone}$.
We are working in the topos $\mathrm{Shv}^\mathrm{fpqc}_{k, \kappa}$, so by the universality of colimits in a topos we get
$$
\mZ\simeq (\varinjlim_i\Spec(A_i))\times_{\mX}\mZ\simeq \varinjlim_i \Spec(A_i)\times_{\mX}\mZ\simeq \varinjlim_i \Spec(B_i).
$$
It follows, since passage to p\'etra spaces commutes with colimits, that $\mZ$ is the p\'etra space of the condensed set $Z \coloneq \varinjlim_i\left|\Spec(B_i)\right |$, with the colimit formed in the category $\mathrm{Cond}_\kappa$.
\end{proof}

\begin{remark}
According to  \cite[Proposition 4.13]{Analytic}, closed embeddings $Z\to X$ into a compactly generated weak Hausdorff space $X$, viewed as a quasi-separated condensed set, correspond precisely to closed embeddings $Z(*)_\mathrm{top}\to X(*)_\mathrm{top}$ in the usual topological sense.  The same thus holds by Proposition \ref{Closed prop} for closed immersions into the p\'etra space $X^\mathrm{p\acute{e}t}_k$ over a finite field $k$.
\end{remark}

\begin{remark}
The condition that a condensed set $X$ is quasi-separated is equivalent to the diagonal map $\Delta :X\to X\times X$ being a closed embedding.
Consider its p\'etra space $\mX = X^\mathrm{p\acute{e}t}_k$  over a finite field $k$.
It follows from Proposition \ref{Closed prop} that the diagonal map $\Delta :\mX\to\mX\times \mX$ is a closed immersion. Thus the fpqc sheaf $\mX$ is not just quasi-separated, but is in fact \textit{separated}. 
\end{remark}

\subsection{Digression: open embeddings of condensed sets}\label{Subsection 3.4}
We would next like to discuss the behavior of open subsets under Stone duality. However, the notion of open embeddings of condensed sets seems not to have been considered previously.
Presumably, that is because the definition of a condensed structure, unlike that of a topology, does not rely on them.
We must
therefore spend this subsection defining this notion and comparing it with the point-set topological notion.

On the algebro-geometric side, open immersions are tied to the Zariski topology, and are well-understood. Because the property of being an open immersion is fpqc-local on the base by \cite[\href{https://stacks.math.columbia.edu/tag/02L3}{Lemma 02L3}]{stacks-project}, it can be extended to morphisms of arbitrary fpqc sheaves via descent. The following definition in the condensed world is given by direct analogy with this algebro-geometric situation.

\begin{definition}\label{Def of open imm}
A morphism of condensed sets $f:X\to Y$ is an  \textit{open embedding} if for every morphism from a profinite set $S\to X$, the pullback $X\times_Y S\to S$ is the (map of condensed set corresponding to) an open embedding of topological spaces.
\end{definition}

\begin{remark}
Because of the fact \cite[Warning 2.14]{Condensed} that not all topological spaces embed into condensed sets, caution must be exercised when attempting to define notions in condensed mathematics via point-set topology. The reason this is not an issue with Definition \ref{Def of open imm} is that compactly generated $T_1$ topological spaces do in fact embed fully faithfully into condensed sets, and are closed under taking open subspaces. Thus even though an open subspace of a profinite set might no longer be profinite itself, we may still view it as a condensed set.
\end{remark}

To justify the above definition of an open embedding as sensible, we will show that it restricts to the familiar notion of the same name on topological spaces. We will accomplish this in Proposition \ref{Open = open in top and cond}, after a preliminary digression on complements.

\begin{definition}
For an injective morphism of $\kappa$-condensed sets $Y\to X$, its \textit{complement} $X-Y:\mathrm{ProFin}_\kappa^\mathrm{op}\to\mathrm{Set}$ is defined by
\begin{equation}\label{formula for complement}
(X-Y)(S) = \{x\in X(S) : S\times_X Y = \emptyset\}
\end{equation}
 for any $\kappa$-small profinite set $S$.
This functor clearly satisfies descent with respect to the effective epimorphism topology, and therefore defines an $\kappa$-condensed set $X-Y\in \mathrm{Cond}_\kappa$ with a canonical morphism $(X-Y)\to X$.
\end{definition}

\begin{remark}\label{Complements are topological}
Consider a  continuous embedding of $\kappa$-compactly generated topological spaces $Y\to X$. Since the canonical functor $\mathrm{Top}^{\kappa-\mathrm{cg}}\to \mathrm{Cond}_\kappa$ is a right adjoint by \cite[Proposition 1.7]{Condensed}, it preserves limits and in particular pullbacks. This implies that the complement $X-Y$ of the corresponding eponymous map of $\kappa$-condensed sets is again (the $\kappa$-condensed set associated to) a topological space, and in fact agrees with the usual point-set notion of the complement $X-Y$.
\end{remark}

\begin{remark}
Recall from Proposition \ref{up the beanstalk} that the left Kan extension $\mathrm{Cond}_\kappa\to\mathrm{Cond}$, through which we view $\kappa$-condensed sets such as $X$, $Y$, and $X-Y$ as condensed sets, preserves colimits and finite limits. It follows that the formula \eqref{formula for complement} for $(X-Y)(S)$ remains valid for any (not-necessarily $\kappa$-small) profinite set $S$.
\end{remark}

It is clear from the definition of a complement \eqref{formula for complement} that for any morphism of condensed sets $X'\to X$, there is a natural isomorphism
$$
X'\times_{X}(X-Y)\simeq X'-(X'\times_X Y'),
$$
which is one sense in which the formation of complements is functorial in $X$. It is also contravariantly functorial in $Y\in \mathrm{Cond}_{/X}$, in the sense that any morphism $Y'\to Y$ induces a natural map of condensed sets
$
(X - Y)\to (X- Y').
$

Equipped with the technology of complements, we can use them to pass between open and closed embeddings in the expected manner.

\begin{prop}\label{Clomplements}
Let
$$\mathrm{Cond}_{/X}^\mathrm{open}, \mathrm{Cond}_{/X}^\mathrm{closed}\subseteq \mathrm{Cond}_{/X}
$$
denote the full subcategories spanned respectively by open and closed embeddings into a fixed condensed set $X$. The passage to complements
$$
(Y\to X)\mapsto ((X- Y)\to X)
$$
defines an anti-equivalence of categories
$$
\mathrm{Cond}_{/X}^\mathrm{open}\simeq(\mathrm{Cond}_{/X}^\mathrm{closed})^\mathrm{op}.
$$
\end{prop}

\begin{proof}
Note first that if $X=S$ is a profinite set, then open and closed embeddings in the sense of condensed sets coincide with those of topological spaces. Likewise does passage to complements recover its usual meaning, and so the statement in question is a basic (definitional) point of point-set topology. In the rest of this proof, we show how this implies the desired result on the level of arbitrary condensed sets.

We must first show that if $Y\to X$ is an open or closed embedding, then $(X-Y)\to X$ is a closed or open embedding respectively.
For any fixed map from a profinite set $S\to X$ we have by the abstract properties of complements a canonical isomorphism $$S\times_X (X-Y)\simeq S- (S\times_X Y),$$
compatible with the structure maps into $S$. Definition \ref{Def of open imm} and Lemma \ref{Lemma on closed condensed} guarantee that the map $S\times_X Y\to S$ is respectively an open or closed embedding of topological spaces. Thus $S-(S\times_X Y)$ corresponds to a respectively closed or open embedding into the profinite set $S$, proving that $(X-Y)\to X$ is an embedding of the desired sort.

Thus passage to complements defines functors
$$
\mathrm{Cond}_{/X}^\mathrm{open}\leftrightarrows(\mathrm{Cond}_{/X}^\mathrm{closed})^\mathrm{op},
$$
and it remains to show that these are inverse to each other.
Without loss of generality, let us consider the case where we start with an open immersion $U\to X$ - the case for a closed immersion is analogous. Let $Z:=X-U$ be the corresponding closed immersion. It follows from its definition that $U\times_X Z=\emptyset$ and so the open embedding $U\to X$ factors through $(X-Z)\to X$. We wish to show that the so-obtained morphism of condensed sets $U\to (X-Z)$ is an isomorphism. Since this is a morphism in the overcateogry $\mathrm{Cond}_{/X}$, it suffices to show that the pullback
$$
U\times_X S\to (X-Z)\times_X S= S - (Z\times_XS)
$$
is an isomorphism for an arbitrary map from a profinite set $S\to X$.
Let $V:=U\times_X S$ and $K = Z\times_X S$. By the open and closed embedding properties, these may be identified with a complementary open subset $V\subseteq S$ and closed subset $K\subseteq S$. The map in question
therefore amounts to $V\to S-K$, which is an isomorphism. That completes the proof.
\end{proof}

\begin{prop}\label{Open = open in top and cond}
Let $X$ be a compactly generated $T_1$ topological space. Under the fully faithful embedding of compactly generated $T_1$ spaces into condensed sets, open embeddings of condensed sets into $X$ coincide with open embeddings of topological spaces into $X$.
\end{prop}

\begin{proof}
An open embedding of topological spaces into $X$ evidently satisfies Definition \ref{Def of open imm}. What therefore needs to be proved is that if $U\to X$ is an open embedding of condensed sets, then $U$ is in fact (the condensed set corresponding to) a topological space, and that the corresponding continuous map $U\to X$ is an open embedding in the usual sense.

Thus let $U\to X$ be an open embedding of condensed sets in the sense of Definition \ref{Def of open imm}.
Let $Z\coloneq X-U$ denote the complementary closed embedding of condensed sets in the sense of Proposition \ref{Clomplements}. By \cite[Proposition 4.13]{Analytic} the condensed subset $Z\subseteq X$ is determined by the closed subspace $W\coloneq Z(*)_\mathrm{top}\subseteq X$ as
$$
Z(S) = X(S)\times_{\mathrm{Hom}_\mathrm{Top}(S, X(*)_\mathrm{top})}\Hom_{\mathrm{Top}}(S, W).
$$
Since we have by the definition of the condensed set associated to a topological space that
$$X(S) = \Hom_{\mathrm{Top}}(S, X(*)_\mathrm{top}),$$
it follows that $Z = W$ is also a topological space itself.
Its (condensed-set-level) complement is by Proposition \ref{Clomplements} the original condensed set $U = X-Z$, which Remark \ref{Complements are topological} identifies with the point-set-topological complement $X-W$. In particular, it is a topological space.
\end{proof}

\begin{remark}
If we wished to dispense with the $T_1$ separation assumption, we could instead consider the case when $X$ is a $\kappa$-compactly generated space for some strong limit cardinal $\kappa$. Open embeddings into $X$ in the category $\mathrm{Cond}_\kappa$ then correspond to topological open embeddings into $X$ in the usual sense. Of course, the standard caveats
\cite[Warning 2.14]{Condensed} apply, and so this correspondence is not necessarily compatible with the passage $\mathrm{Cond}_\kappa\to \mathrm{Cond}$ without a $T_1$ assumption.
\end{remark}

\subsection{Open embeddings and open immersions}\label{Subsection 3.5}
We spent the previous subsection introducing, studying, and justifying the notion of open embeddings of condensed sets. Now we compare it with open immersions of fpqc stacks under Stone duality.

\begin{prop}\label{Open prop}
Let  $\mX$ be the p\'etra space of a quasi-separated condensed set
$X$. Stone duality of Theorem \ref{FF on topoi} induces a bijection between open embeddings of condensed sets $U\to X$ and open immersions of fpqc sheaves $\mU\to \mX$.
\end{prop}

\begin{proof}
Since $X$ is quasi-separated, its p\'etra space $\mX$ admits an fpqc cover by spectra of Stone $k$-algebras.
The property being an open immersion is fpqc-local on the base \cite[\href{https://stacks.math.columbia.edu/tag/02L3}{Lemma 02L3}]{stacks-project}, and so a map of fpqc sheaves $\mU\to \mX$ is an open immersion if and only if its pullback 
$\Spec(A)\times_{\mX}\mU$ along any map $\Spec(A) \to\mX$ for $A\in\CAlg_k^\mathrm{Stone}$
is an open immersion into $\Spec(A)$.
In light of Proposition \ref{Stone space in terms of pi_0}, it is therefore clear that a morphism of condensed sets $U\to X$ is an open embedding in the sense of Definition \ref{Def of open imm} if and only if the corresponding map on p\'etra spaces $\mU\to \mX$ is an open immersion.

To conclude the proof, it remains to show that any open immersion  $\mU\to \mX$ into the p\'etra space $\mX=X^\mathrm{p\acute{e}t}_k$ is in fact a map of p\'etra spaces; that is to say, that $\mU$ is a p\'etra space itself. To see that, write
$$
\mX = \varinjlim_i\Spec(A_i)
$$
for a small diagram of Stone $k$-algebras $A_i$. Then
$$
\mU_i\coloneq\Spec(A_i)\times_{\mX}\mU
$$
is an open subscheme of $\Spec(A_i)$, which is precisely determined by the open subset
$$
U_i\coloneq \left |\mU_i\right |\subseteq \left |\Spec(A_i)\right |.
$$
Given the Zariski open $U_i,$ a map of schemes factors through the inclusion $\mU\subseteq \Spec(A)$ if and only if its underlying continuous map factors though
$
U_i\to\left |\Spec(A)\right |.
$
The topological space $U_i$ is an open subset of a compactly generated Hausdorff space, implying that it satisfies those properties itself. Its p\'etra space can therefore be described as in Proposition \ref{FOP for Haus}, which we now recognize as the just-described mapping property for $\mU_i$. That is to say, we find that $\mU_i = U_i^\mathrm{p\acute{e}t}$ are p\'etra spaces, and thanks to the faithful flatness of the p\'etra space functor by Theorem \ref{FF on topoi}, these even fit into a  diagram of condensed sets $i\mapsto U_i$.
Now we use that colimits in a topos (such as $\mathrm{Shv}_{k,\kappa}^\mathrm{fpqc}$, to which we could have reduced via strandard argument of picking $\kappa$) are universal
 to write
$$
\mU\simeq (\varinjlim_i\Spec(A_i))\times_{\mX}\mU\simeq \varinjlim_i \Spec(A_i)\times_{\mX}\mU\simeq \varinjlim_i \mU_i.
$$
Setting $U\coloneq\varinjlim_i U_i\in \mathrm{Cond}$, it now follows from the p\'etra space functor preserving colimits that $\mU\simeq U^\mathrm{p\acute{e}t}_k$, completing the proof.
\end{proof}

\begin{remark}
With the additional assumption of a finite ground field $k$,  Proposition \ref{Open prop} could be derived by passage to complements and Proposition \ref{Clomplements} from the corresponding result on closed immersions Proposition \ref{Closed prop}. The reason we instead decided to spell out the (very similar) proof, is that we could make it work without the finiteness assumption on $k$. That is because the key ingredient in the proof of Proposition \ref{Open prop} is the special description Proposition \ref{FOP for Haus}, as opposed to Lemma \ref{Lemma on closed immersions} which plays a similar role in the proof of Proposition \ref{Closed prop}, and which is where the finiteness assumption gets used.
\end{remark}

One surprising aspect of Proposition \ref{Open prop} is that the Zariski topology, which is usually viewed as too coarse and unrefined, is, in fact, capable of encoding nice topological spaces. This can be made explicit by recalling the construction of the \textit{underlying topological space} $\left |\mX\right |$ of any fpqc stack $\mX$: its points consist of field-points $\mX(L)$ for field extensions $L/k$, identified up to specialization. It is topologized by declaring the open subsets $\left |\mU\right |$ to be those subsets that arise from open immersions $\mU\to \mX$.

\begin{corollary}
Let  $\mX$ be the p\'etra space of a compactly generated Hausdorff, viewed as a quasi-separated condensed set
$X$. There is a canonical homeomorphism
$
\left |\mX\right |\simeq X.
$
\end{corollary}

\begin{proof}
Since $X$ is compactly generated Hausdorff, its $L$-valued p\'etra space $\mX$ can be described as in Proposition \ref{FOP for Haus}
$$
\mX(L)\simeq \Hom_{\mathrm{Top}}(\left |\Spec(A)\right |, X) = \Hom_{\mathrm{Top}}(*, X) \simeq X
$$
for any field $L$, and so $\left|\mX\right |\simeq X$ is a bijection. It follows from Proposition \ref{Open prop} that the opens $\left |\mU\right |\subseteq \left |\mX\right |$ then arise precisely and uniquely from open subsets $U\subseteq X$, hence the map in question is a homeomorphism.
\end{proof}

\subsection{Sheaves of modules and quasi-coherent sheaves}\label{Subsection 3.6}

We next turn to study how Stone duality extends to the natural categories of sheaves of modules on both sides. In the following, let $\CAlg(\mathrm{LPr})$ denote the category of locally presentable symmetric monoidal categories, with left-adjoint symmetric monoidal functors as morphism.

\begin{lemma}\label{What are sheaves}
Let $k$ be a ring.
There exist functors
$$
 \mathrm{Shv}_{\Mod_k} :\mathrm{Cond}^\mathrm{op}\to \CAlg(\mathrm{LPr})
 \qquad\quad
\mathrm{QCoh} : (\mathrm{Shv}_k^\mathrm{fpqc})^\mathrm{op} \to \CAlg(\mathrm{LPr}),
$$
determined uniquely up to natural equivalence by the following requirements:
\begin{enumerate}[label =(\roman*)]

\item They commute with small filtered limits.\label{WASi}

\item They respectively satisfy effective and fpqc descent. \label{WASii}

\item On affine $k$-schemes and profinite sets respectively, they restrict to the functors $\CAlg_k\ni A\mapsto \Mod_A$ and $\mathrm{ProFin}^\mathrm{op}\ni S\mapsto \mathrm{Shv}_{\Mod_k}(S)$, where the latter denotes\label{WASiii} usual category of sheaves of $k$-modules on the profinite set $S$, when viewed as a totally disconnected topological space.
\end{enumerate}
\end{lemma}

\begin{proof}
Since the functors $\mathrm{Shv}_{\Mod_k}$ and $\QCoh$ are both required to commute with filtered colimits by \ref{WASi}, it suffices to define their restrictions to the $\kappa$-variants $\mathrm{Cond}_\kappa\subseteq\mathrm{Cond}$ and $\mathrm{Shv}_{k, \kappa}^\mathrm{fpqc}\subseteq \mathrm{Shv}_k^\mathrm{fpqc}$ for each strongly inaccessible cardinal $\kappa$, and then pass to the filtered colimit over $\kappa$. Thanks to the descent assumption of \ref{WASii}, defining the respective functors out of the Grothendieck topoi $\mathrm{Cond}_\kappa$ and $\mathrm{Shv}_{k, \kappa}^\mathrm{fpqc}$ reduces to specifying them on  the corresponding sites $\CAlg_{k, \kappa}^\mathrm{op}$ and $\mathrm{ProFin}_\kappa$. There the requirement \ref{WASiii} identifies the two functors as $A\mapsto\Mod_A$ and $S\mapsto\mathrm{Shv}_{\Mod_k}(S)$ respectively. It remains to verify that this definition on sites is compatible with requirements \ref{WASi} and \ref{WASii}. Commuting with filtered limits is well-known in both cases, see e.g.~\cite[Corollary 4.5.1.7]{SAG} for the module case. Descent for sheaves of modules on profinite spaces with respect to the effective epimorphism topology is a special case of \cite[Theorem 0.5, Remark 1.18]{Haine}, while fpqc descent for modules  \cite[\href{https://stacks.math.columbia.edu/tag/023N}{Proposition 023N}]{stacks-project} is due to Grothendieck \cite{Gro}.
\end{proof}

\begin{definition}\label{Def of QCoh and Shv}
The functors $X\mapsto\mathrm{Shv}_{\Mod_k}(X)$ and $\mX\mapsto\QCoh(\mX)$ of Lemma \ref{What are sheaves} are respectively called the
\textit{sheaves of $k$-modules on a condensed set $X$} and \textit{quasi-coherent sheaves on an fpqc sheaf $\mX$ over $k$}.
\end{definition}

\begin{remark}\label{Sheaves are sheaves}
Let $X$ be a locally compact Hausdorff space, viewed as a qcqs condensed set. 
The category $\mathrm{Shv}_{\Mod_k}(X)$ of sheaves of $k$-modules on $X$ in the sense of Definition \ref{Def of QCoh and Shv} was defined in the course of the proof of Lemma \ref{What are sheaves} through effective descent from the case of profinite sets.
By \cite[Theorem 0.5, Remark 1.18]{Haine}, $\mathrm{Shv}_{\Mod_k}(X)$ in this sense is naturally equivalent to the usual category of sheaves of $k$-modules on $X$. The latter is usually defined to consist of those presheaves $\sF:\mathrm{Op}(X)^\mathrm{op}\to \Mod_k$ from the poset  $\mathrm{Op}(X)$ of open subsets in $X$ which satisfy descent with respect to the standard Grothendieck topology on  $\mathrm{Op}(X)$.
\end{remark}

\begin{remark}\label{QCoh vs QCoh}
The above construction of quasi-coherent sheaves is analogous to the more sophisticated derived versions \cite[Section 6.2, Subsection 7.3.3]{SAG} and \cite[Chapter 3]{Dirac2}, which produce the $\i$-category of quasi-coherent sheaves\footnote{In fact, it is common in derived algebraic geometry and related fields to use the notation $\QCoh(\mX)$ for what we are calling $\mathcal D_\mathrm{QCoh}(\mX)$. Since we are working purely in the classical setting, we prefer a variant of the classical algebro-geometric notation.} $\mathcal D_\mathrm{QCoh}(\mX)$ on any derived (or spectral) stack $\mX$. These carry a well-behaved $t$-structure when $\mX$ is a geometric stack. As noted in Remark \ref{Remark vs geometric stacks}, geometric spaces in the sense of Definition \ref{Def of geom space} are special cases of geometric stacks. The category of quasi-coherent sheaves in the sense of Definition \ref{Def of QCoh and Shv} are in that case related to the $\infty$-category $\mathcal D_\mathrm{QCoh}(\mX)$ as
$$
\QCoh(\mX)\simeq\mathcal D_\mathrm{QCoh}(\mX)^\heartsuit,
$$
i.e.\ it is the heart of its natural $t$-structure.
\end{remark}

We return to Stone duality, and therefore to the standing assumption that $k$ is a field.

\begin{theorem}\label{Shv = QCoh}
For any condensed set $X$, there is a canonical equivalence of symmetric monoidal categories
$$
\QCoh(X^\mathrm{p\acute{e}t}_k)\simeq \mathrm{Shv}_{\Mod_k}(X),
$$
natural in $X$.
More precisely, the composite functor
$$
\mathrm{Cond}^\mathrm{op}\xrightarrow{(-)^\mathrm{p\acute{e}t}_k}(\mathrm{Shv}_k^\mathrm{fpqc})^\mathrm{op}\xrightarrow{\QCoh}\CAlg(\mathrm{LPr})
$$
is naturally equivalent to the functor $\mathrm{Shv}_{\Mod_k}$.
\end{theorem}

\begin{proof}
In light of how the functors $\QCoh$ and $\mathrm{Shv}_{\Mod_k}$ are defined in Lemma \ref{What are sheaves}, and using the fact that the p\'etra space functor $(-)^\mathrm{p\acute{e}t}_k=f^*$ commutes with all colimits by Theorem \ref{FF on topoi}, it suffices to show that the functors $\CAlg_k^\mathrm{Stone}\ni A\mapsto \Mod_A\in\CAlg(\mathrm{LPr})$ and $\mathrm{ProFin}^\mathrm{op}\ni S\mapsto \mathrm{Shv}_{\Mod_k}(X)\in\CAlg(\mathrm{LPr})$ are naturally identified under Stone duality in the form of Theorem \ref{Stone theorem}. That is the content of Proposition \ref{Global functions} and the next lemma.
\end{proof}

\begin{lemma}
Let $S$ a profinite set. The global sections functor $\sF\mapsto\Gamma(S; \sF)$ is an equivalence of $k$-linear  symmetric monoidal categories
$\mathrm{Shv}_{\Mod_k}(S)\simeq \Mod_{\CMcal C(S, k)}.$
\end{lemma}

\begin{proof}
According to Remark \ref{Sheaves are sheaves}, sheaves of $k$-modules on $S$ can be interpreted in the usual way in terms of open subsets of $S$.  Because clopen subsets form a basis for the topology of a profinite set, it suffices to further restrict to those. A sheaf of $k$-modules $\sF$ on $S$ is thus equivalent to a functor $\sF:\mathrm{Clop}(S)^\mathrm{op}\to\Mod_k$ such that its restriction maps induce an isomorphism
$$
\sF(U\cup V) \simeq \sF(U)\oplus \sF(V)
$$
 for any pair of disjoint clopens $U, V\subseteq S$. Since this splitting is canonical, we may in particular identify the sections $\sF(U)$ over any clopen $U\subseteq S$ with a $ k$-linear subspace of the global sections $\sF(S)$ (by extending sections over $U$ to the complement $S-U$ by zero). This  is explicitly given by $\sF(U) = \chi_U\sF(S)$, where the characteristic function $\chi_U\in \CMcal C(S, k)$ is specified by $\chi_U\vert_U =1$ and $\chi_U\vert_{S-U} = 0$.

Now let $M$ be a $\CMcal C(S, k)$-module and consider the functor $\sF:\mathrm{Clop}(S)^\mathrm{op}\to \Mod_k$ defined by $\sF(U)\coloneq\chi_UM$. The functoriality comes from the fact that for an inclusion of clopens $U\subseteq V$ we have $\chi_U\chi_V = \chi_U$, so multiplication by $\chi_U$ gives rise to the restriction morphism
$$
\sF(V) =\chi_VM\xrightarrow{\chi_U\cdot-} \chi_U\chi_VM = \chi_UM=\sF(U).
$$
For a pair of disjoint clopen subsets $U, V\subseteq X$ we have $\chi_U+\chi_V = \chi_{U\cup V}$, from which it follows that $\chi_{U\cup V} M=\chi_U M\oplus \chi_VM$ and so the thus-defined presheaf on clopens is actually a sheaf on $S$. Since $\chi_S = 1$, the global sections of this sheaf are $\sF(S) = 1M = M,$ proving that we have exhibited an inverse to the global sections functor.

We have obtained the desired equivalence of categories. We will verify this equivalence is symmetric monoidal by showing that the above-defined functor $\Mod_{\CMcal C(S, k)}\to \mathrm{Shv}_{\Mod_k}(S)$ preserves tensor products.
 To that end,  note that the  idempotence of characteristic functions $\chi_U^2=\chi_U$ implies the equality
$$
\chi_U(M\o_{\CMcal C(S, k)} M') =\chi_U^2(M\o_{\CMcal C(S, k)}  M') = (\chi_U M)\o_{\CMcal C(S, k)}  (\chi_UM')
$$
for any pair of $\CMcal C(S, k)$-modules $M, M'$. By the inclusion of sections over $U$ into sections over $S$ discussed above, the tensor product on the right can be taken over $\CMcal C(U, k)$ as well. Then the right-hand side in question can be recognized as sections over $U$ of the tensor product $\sF\o_k\sF'$ of the sheaves $\sF$ and $\sF'$ associated to $M$ and $M'$ respectively.
\end{proof}

\begin{remark}
A derived-category analogue of the result of Theorem \ref{Shv = QCoh} over $k=\mathbf Z$
in considered  as ``a somewhat curious case'' in
 \cite[Exerise 1.7]{6FF}
 for $X$ a finite-dimensional locally compact Hausdorff space. 
There, the associated algebro-geometric functor to $X$ is obtained by sending an affine scheme $S$ to
$$
\mathrm{Aff}_k^\mathrm{op}\ni S\mapsto\Hom_\mathrm{Top}(\left |S\right|, X)\in \mathrm{Set}.
$$
Up to the difference of working over the absolute base $\mathbf Z$ as opposed to a field $k$, this agrees with the p\'etra space $X^\mathrm{p\acute{e}t}_k$ by Proposition \ref{FOP for Haus}.
 \end{remark}

\begin{remark}\label{Tout est plat}
Any quasi-coherent sheaf on the p\'etra space of a compact Hausdorff space over a field $k$ is flat.
Here recall that a quasi-coherent sheaf $\sF\in\QCoh(\mX)$ is flat if its restriction $x^*(\sF)$ along any morphism $x:\Spec(A)\to \mX$ is  a flat $A$-module.
Since flatness is a local property with respect to the fpqc topology \cite[\href{https://stacks.math.columbia.edu/tag/036K}{Lemma 036K}]{stacks-project}, it may be checked on fpqc covers. For any compact Hausdorff space, a continuous surjection from a profinite set gives rise to an fpqc cover $\Spec(A)\to \mX$ of its p\'etra space $\mX$ by the spectrum of a Stone $k$-algebra $A$. The flatness claim now follows from the absolute flatness of Stone algebras Proposition \ref{Stone implies abs flat}, since all of their modules are flat. Flatness of all quasi-coherent sheaves seems like a rather unusual behavior from the perspective of the algebro-geometric intuition built upon experience with schemes or algebraic spaces; but according to Proposition \ref{non-algebraic}, that has little bearing on p\'etra spaces.
\end{remark}

\subsection{Tannaka reconstruction of compact Hausdorff spaces}\label{Subsection 3.7}
In Proposition  \ref{Stone space in terms of pi_0} we described the p\'etra space of a condensed set in terms of fpqc sheafification, and we could only simplify it for profinite sets.  We use Theorem \ref{Shv = QCoh} and the recent strengthened form of Lurie's Tannaka reconstruction due to \cite{German} (really a relative version of it, deduced from Stefanich's result below as Lemma \ref{German's coup}) to give an explicit point-wise description of the p\'etra space of any compact Hausdorff space. In the following, let us denote  the category of $k$-linear locally presentable categories by  $\mathrm{LPr}_k\coloneq \Mod_{\Mod_k}(\mathrm{LPr})$.

\begin{lemma}[Tannaka Reconstruction, {\cite{German}}]\label{German's coup}
Let $\mX$ and $\mY$ be geometric spaces over a ring $k$. Pullback of quasi-coherent sheaves induces a canonical bijection
$$
\Hom_{\mathrm{Shv}_k^\mathrm{fpqc}}(X, Y)\simeq \Hom_{\CAlg(\mathrm{LPr}_k)}(\QCoh(Y), \QCoh(X)).
$$
\end{lemma}

\begin{proof}
Observe that the geometric spaces in the sense of Definition \ref{Def of geom space} are special cases of quasi-compact classical geometric stacks with an affine diagonal in the sense of \cite[Subsection 1.1]{German}. In light of that, the analogous absolute (which is to say working over the absolute base $\mathbf Z$ instead of over the chosen base ring $k$) version of the desired result
\begin{equation}\label{German's Tannaka}
\Hom_{\mathrm{Shv}_{\mathbf Z}^\mathrm{fpqc}}(\mX, \mY)\simeq \Hom_{\CAlg(\mathrm{LPr})}(\QCoh(\mY), \QCoh(\mX))
\end{equation}
is proved in \cite[Theorem 1.0.3]{German}. Since it is given by pullback of quasi-coherent sheaves, this equivalence is clearly natural in $\mY$. To prove the relative variant over $k$, it therefore suffices to express the relative Hom sets over $k$ on both sides in terms of the absolute Hom sets in  an analogous way. 
Under the standard identification $\mathrm{Shv}_k^\mathrm{fpqc}\simeq (\mathrm{Shv}_{\mathbf Z}^\mathrm{fpqc})_{/\Spec(k)}$, the canonical maps form a pullback square
$$
\begin{tikzcd}
\Hom_{\mathrm{Shv}_{k}^\mathrm{fpqc}}(\mX, \mY)\ar[d]\ar[r] & \{p\}\ar[d]\\
\Hom_{\mathrm{Shv}_{\mathbf Z}^\mathrm{fpqc}}(\mX, \mY)\ar[r] & {\Hom_{\mathrm{Shv}_{\mathbf Z}^\mathrm{fpqc}}(\mX, \Spec(k))}
\end{tikzcd}
$$
where $p:\mX\to \Spec(k)$ is final morphism in fpqc sheaves over $k$. Similarly we have standard equivalences of categories of algebras $\CAlg(\mathrm{LPr}_k)\simeq \CAlg_{\Mod_k}(\mathrm{LPr}_k)\simeq \CAlg(\mathrm{LPr}_k)_{\Mod_k/}$, and so pullback of quasi-coherent sheaves induces a diagram
$$
\begin{tikzcd}
 \Hom_{\CAlg(\mathrm{LPr}_k)}(\QCoh(\mY), \QCoh(\mX))\ar[d] \ar[r]&   \{p^*\}\ar[d] \\  
 {\Hom_{\CAlg(\mathrm{LPr})}(\QCoh(\mY), \QCoh(\mX))}
 \ar[r] & \Hom_{\CAlg(\mathrm{LPr})}(\Mod_k, \QCoh(\mX))
\end{tikzcd}
$$
which is likewise a pullback square. Since everything in the two squares is mapped to the other by the natural isomorphism \eqref{German's Tannaka}, we obtain the desired result.
\end{proof}

\begin{prop}\label{Tannaka rec}
Let $X$ be a compact Hausdorff space. Its p\'etra space can be canonically identified as
$$
X^\mathrm{p\acute{e}t}_k(A) \simeq \Hom_{\CAlg(\mathrm{LPr}_k)}(\mathrm{Shv}_{\Mod_k}(X), \Mod_A)
$$
for any $k$-algebra $A$.
\end{prop}

\begin{proof}
Since p\'etra spaces of compact Hausdorff spaces, viewed as qcqs condensed sets, are geometric spaces by Theorem \ref{condensed is geometric}, and there are natural $k$-linear symmetric monoidal equivalences $\QCoh(X^\mathrm{p\acute{e}t}_k)\simeq \mathrm{Shv}_{\Mod_k}(X)$ by Theorem \ref{Shv = QCoh}  and $\QCoh(\Spec(A))\simeq \Mod_A$ by the definition of quasi-coherent sheaves, the desired result follows by a direct application of Lemma \ref{German's coup}.
\end{proof}

\begin{remark}
As a consequence of Remark \ref{Tout est plat}, we can replace the module category $\Mod_A$ in the statement of Theorem \ref{Tannaka rec} with the full subcategory $\Mod_A^\flat\subseteq \Mod_A$ of flat modules. That is essentially the context of the  Tannaka reconstruction result \cite[Theorem 9.7.0.1]{SAG}.
\end{remark}

As a special case, we obtain the following result that is surely well-known:

\begin{corollary}\label{Shv is ff on CHaus}
Let $X$ and $Y$ be a pair of compact Hausdorff spaces. 
Pullback of sheaves of $k$-modules induces a natural bijection
$$
\Hom_{\mathrm{CHaus}}(X, Y)\simeq \mathrm{Hom}_{\CAlg(\mathrm{LPr}_k)}(\mathrm{Shv}_{\Mod_k}(Y), \mathrm{Shv}_{\Mod_k}(X)).
$$
In particular, if the categories of sheaves of $k$-modules $\mathrm{Shv}_{\Mod_k}(X)$ and $\mathrm{Shv}_{\Mod_k}(Y)$ are symmetrically monoidally equivalent, then $X$ and $Y$ are homeomorphic.
\end{corollary}

\begin{proof}
By applying Lemma \ref{German's coup} in light of Theorem \ref{Shv = QCoh}, we find that
$$
 \Hom_{\CAlg(\mathrm{LPr}_k)}(\mathrm{Shv}_{\Mod_k}(Y), \mathrm{Shv}_{\Mod_k}(X))\simeq \Hom_{\mathrm{Shv}_k^\mathrm{fpqc}}(X_k^\mathrm{p\acute{e}t}, Y_k^\mathrm{p\acute{e}k}).
$$
The rest follows from the functor $X\mapsto X^\mathrm{p\acute{e}t}_k$ being fully faithful by Theorem \ref{FF on topoi}.
\end{proof}

\begin{remark}
For any $\mC\in \mathrm{CAlg}(\mathrm{LPr}_k))$, we define its \textit{Tannakian spectrum} to be the functor
$\mathrm{Spec}^\mathrm T_k(\mC) :\CAlg_k\to \mathrm{Set}$
given by
$$
A\mapsto \Hom_{\CAlg(\mathrm{LPr}_k)}(\mC, \Mod_A).
$$
Since $A\mapsto\Mod_A$ satisfies fpqc descent, this is an fpqc sheaf.
The Tannaka Reconstruction Theorem, in the form of Lemma \ref{German's coup}, then asserts that the composite functor 
$$
\mathrm{GeomSpc}_k\xrightarrow{\QCoh}\CAlg(\mathrm{LPr}_\kappa)^\mathrm{op}\xrightarrow{\mathrm{Spec}_k^\mathrm{T}}\mathrm{Shv}_k^\mathrm{fpqc}
$$
is a factorization of the subcategory inclusion $\mathrm{GeomSpc}_k\subseteq \mathrm{Shv}_k^\mathrm{fpqc}$ of geometric spaces into all fpqc sheaves, and in particular fully faithful. This perspective offers a reinterpretation of the p\'etra space functor.
For any fixed inaccessible cardinal $\kappa$, Lemma \ref{Shv is ff on CHaus} guarantees that the  functor
$$
\mathrm{CHaus}_\kappa \xrightarrow{\mathrm{Shv}_{\Mod_k}}\CAlg(\mathrm{LPr}_k)^\mathrm{op} \xrightarrow{\mathrm{Spec}^\mathrm{T}_k} \mathrm{Shv}_{k, \kappa}^\mathrm{fpqc}
$$
is fully faithful as well. By passage to the category of sheaves on the left-hand side and the equivalence of categories $\mathrm{Cond}_\kappa\simeq \mathrm{Shv}(\mathrm{CHaus}_\kappa, \tau_\mathrm{eff})$  from \cite[Proposition 2.3]{Condensed}, this extends to a geometric morphism
$$
\mathrm{Shv}_{k, \kappa}^\mathrm{fpqc}\to
\mathrm{Cond}_\kappa,
$$
compatible with left Kan extensions along $\kappa<\kappa'$. Passing to the colimit over all inaccessible cardinals $\kappa$, we obtain a functor $\mathrm{Cond}\to \mathrm{Shv}_{k, \kappa}^\mathrm{fpqc}$, which is precisely the p\'etra space functor $X\mapsto X^\mathrm{p\acute{e}t}_k$. Said functor therefore arises as the descent-induced extension of the Tannaka spectrum of sheaves of $k$-modules from compact Hausdorff spaces.
\end{remark}

\section{Variants of condensed Stone duality}\label{Appendix}

In this final section, we collect two variants of the Stone duality story developed in this paper. The first, discussed in Subsection \ref{Subsection 4.1}, is a simple modification of the story above to the ``light condensed" context. The second, and the topic of Subsection \ref{Subsection 4.2}, concerns the largely conjectural extension of the results of this paper to the $\i$-categorical setting.

\subsection{Light condensed sets}\label{Subsection 4.1}
In their recent work on analytic geometry \cite{CS23}, Clausen and Scholze have introduced a slightly modified ``light'' setting for condensed mathematics. It circumvents some of the finer set-theoretical points of the previous non-light approach and is well-suited for foundations of analytic geometry. Its starting point is:

\begin{definition}
A profinite set $S$ is \textit{light} if it is a countable filtered colimit of finite sets inside $\mathrm{ProFin}$. The full subcategory $\mathrm{ProFin}^\mathrm{light}\subseteq\mathrm{ProFin}$ inherits the effective epimorphism topology. The category of \textit{light condensed sets} is the corresponding topos
$$
\mathrm{Cond}^\mathrm{light} \coloneq\mathrm{Shv}(\mathrm{ProFin}^\mathrm{light}, \tau_{\mathrm{eff}}).
$$
\end{definition}

For a fixed cardinal $\kappa$, recall from Section \ref{Section 2} the full subcategories 
$$
\CAlg_{k, \kappa}^\mathrm{Stone}\subseteq\CAlg_{k,  \kappa}\subseteq\CAlg_k
$$
of all $\kappa$-small (Stone) $k$-algebras, i.e.\ $k$-algebras $A$ which satisfy $|A|< \kappa$. Recall also that the countable cardinal $\aleph_0$ has a \textit{successor cardinal}
$\aleph_1$, such that any set is by definition countable if and only if it is $\aleph_1$-small. The following observation, phrased in terms of countability, is due to \cite[Lecture 2]{CS23}

\begin{prop}\label{Stone light def}
Let $k$ be a finite field. The Stone duality equivalence of Theorem \ref{Stone theorem} restricts to an equivalence of categories
$$
\CAlg^\mathrm{Stone}_{k, \aleph_1}
\simeq
(\mathrm{ProFin}^\mathrm{light})^\mathrm{op}.
$$
\end{prop}

\begin{proof}
Any profinite set $S$ and any Stone $k$-algebra $A$ may be written respectively in the form $S=\varprojlim_{i\,\in\, \mathcal I} S_i$ and $A=\varinjlim_{j\,\in \, \mathcal J} A^{S_j}$ for some filtered indexing categories $\mathcal I, \mathcal J$ and diagrams of finite sets $S_i, S_j$. By Theorem \ref{Stone theorem}, Stone duality sends a profinite set $S$ to the Stone $k$-algebra
$$
A = \CMcal C(S, k) = \varinjlim_{i\,\in\,\mathcal I}\,\CMcal C(S_i, k)
$$
and a Stone $k$-algebra $A$ to the profinite set
$$
\Hom_{\CAlg_k}(A, k)=\varprojlim_{j\, \in\,\mathcal J}\Hom_{\CAlg_j}(A^{S_j}, k).
$$
Since $S_i, S_j,$ and $k$ are all finite sets, we see that $\mathcal I$ is countable if and only if $\mathcal J$ is.
\end{proof}

\begin{remark}
The conclusion of Proposition \ref{Stone light def} might at first glance seem to be at odds with the $\kappa$-small variant of Stone duality
$$
\CAlg^\mathrm{Stone}_{k, \kappa}\simeq \mathrm{ProFin}_\kappa^\mathrm{op}
$$
 from the proof of Proposition \ref{FF on topoi}.The key difference is that there, we were assuming that $\kappa$ is an inaccessible cardinal, and in particular a strong limit cardinal. That means that $\kappa'<\kappa$ implies $2^{\kappa'}<\kappa$. This fails for the cardinal $\aleph_1$, since while $\aleph_0<\aleph_1$ holds by definition, it follows from Cantor's diagonalization argument that $\aleph_1\le 2^{\aleph_0}$ (whether or not this is an equality, i.e.\ independently of the continuum hypothesis). It is because of this that $\mathrm{ProFin}^\mathrm{light}\supset \mathrm{ProFin}_{\aleph_1}$ is a strict subset. For instance, the Cantor set $\{0, 1\}^{\aleph_0}$ is a light profinite set with the cardinality of the continuum.
\end{remark}

With the interaction with algebra-profinite-level Stone duality established, it is simple to formulate a light version of Theorem \ref{FF on topoi}, i.e.\ to extend light Stone duality from light profinite sets to light condensed sets.

\begin{theorem}\label{light FF}
Let $k$ be a finite field.
The composite of Stone duality of Theorem \ref{Stone theorem} and the subcategory inclusion
$$
\mathrm{ProFin}^\mathrm{light}\simeq (\CAlg^\mathrm{Stone}_{k, \aleph_1})^\mathrm{op}\subseteq \CAlg_{k, \aleph_1}^\mathrm{op}
$$
gives rise to an adjunction
$$
f^* :\mathrm{Cond}^\mathrm{light}\rightleftarrows \mathrm{Shv}_{k, \aleph_1}^\mathrm{fpqc} :f_*
$$
whose left adjoint $f^*$ is fully faithful and commutes with all finite limits.
\end{theorem}

\begin{proof}
Because Proposition \ref{Stone light def} shows that light profinite sets are compatible with Stone duality, most of the proof of Theorem \ref{FF on topoi} applies to this context. For instance, we need to observe that the pearl functor restricts to a right adjoint $(-)^\circ : \CAlg_{k, \aleph_1}\to\CAlg_{k, \aleph_1}^\mathrm{Stone}$  to the forgetful functor; but that follows from Proposition \ref{Pearl as a subring}.
\end{proof}

\begin{remark}
Unlike Theorem \ref{FF on topoi} which holds over any field $k$, its light version Theorem \ref{light FF} requires working over a finite field. That is because the finiteness assumption is crucial for Proposition \ref{Stone light def}, the light version of algebra-level Stone duality.
\end{remark}

\begin{remark}
Under left Kan extension $\mathrm{Shv}_{k, \aleph_1}^\mathrm{fpqc}\to\mathrm{Shv}^\mathrm{fpqc}_k$ (viewed as a subcategory of accessible presheaves on $\CAlg_k$), ``light p\'etra spaces'' are sent to those p\'etra spaces of condensed sets which are left Kan extended from countable $k$-algebras.
That
might seem very restrictive, but since we are working over a finite field, any finitely presented $k$-algebra is automatically countable. It follows that any fpqc sheaf $\mX :\CAlg_k\to \mathrm{Set}$ which is locally of finite presentation, in the sense that it preserves all small filtered colimits, belongs to $\mathrm{Shv}_{k, \aleph_1}^\mathrm{fpqc}$. As such, passing from all condensed sets to light condensed sets is formally analogous to restricting attention only to algebro-geometric objects which are locally of finite presentation\footnote{For a more perfect analogy, we should really use the ``locally of countable presentation'' assumption, but that is not as common in algebraic geometry.} and reaping the benefits of the set-theoretic simplifications and other nuances that this affords in setting up the foundations.
\end{remark}

\subsection{Towards higher Stone duality}\label{Subsection 4.2}
All discussion in this paper so far has been limited to the 1-categorical case. But both sides of Stone duality, i.e.\ both condensed set and fpqc sheaves over a field $k$, admit natural $\i$-categorical extensions. These are:
\begin{itemize}
\item The $\i$-category $\mathrm{Cond}(\mS)$ of \textit{condensed anima\footnote{Also known as $\i$-groupoids, spaces, weak homotopy types, etc.}}  in the sense of \cite{Condensed} (i.e.\ up to set-theoretic nuances, the \textit{pyknotic spaces} of \cite{Pyknotic}), consisting of those accessible functors $X:\mathrm{ProFin}^\mathrm{op}\to\mS$ which satisfy hyperdescent with respect to the effective epimorphism topology.

\item The $\i$-category $\mathrm{Stack}_k^\mathrm{fpqc}$ of \textit{$\i$-groupoid-valued fpqc stacks\footnote{These are \textit{higher stacks}, but not \textit{derived stacks}; their functors of points are valued in $\i$-groupoids, but are defined on the usual 1-category of $k$-algebras as opposed to the $\i$-category of animated $k$-algebras.}  over $k$}, consisting of those accessible functors $\mX:\CAlg_k\to\mS$ which satisfy fpqc hyperdescent.
 \end{itemize}

 \noindent Just as in the proof of Theorem \ref{FF on topoi},  Proposition \ref{site preservation lemma} extends the profinite-level Stone duality of Theorem \ref{Stone theorem} to a canonical $\i$-categorical adjunction
\begin{equation}\label{Higher}
f^*:\mathrm{Cond}(\mS)\rightleftarrows \mathrm{Stack}_k^\mathrm{fpqc}:f_*,
\end{equation}
whose left adjoint preserves finite limits. 

\begin{quest}
Is the left adjoint $f^*$ of \eqref{Higher} fully faithful?
\end{quest}

\noindent We can attempt to imitate the proof of the corresponding set-valued analogue Theorem \ref{FF on topoi} in this setting. It reduces answering the question in the affirmative to:

\begin{quest}
Does the pearl functor $(-)^\circ :\CAlg_k\to \CAlg_k$, equivalent by Proposition \ref{Pearl as pi_0} to $A\mapsto \pi_0(\left |\Spec(A)\right |)$, preserve fpqc (co)hypercovers of any Stone $k$-algebra?
\end{quest}

\noindent Unfortunately, we do not know if this is true. An affirmative answer would, in conjunction with \cite[Example 3.3.10]{Pyknotic}, give a fully faithful embedding of pro-$\pi$-finite $\i$-groupoids into fpqc stacks
$$
\mathrm{Pro}(\mS^\pi)\hookrightarrow\mathrm{Cond}(\mS)\xrightarrow{f^*}\mathrm{Stack}_k^\mathrm{fpqc}.
$$
Depending on $k$, this might land inside the subcategory of Toën's \textit{affine $k$-stacks} \cite{Toen} (called \textit{coaffine stacks} in \cite[Section 4]{DAG8}), which should then be related to Sullivan-type ``rational homotopy theory'' results as discussed in \cite{Antieau}.

\end{document}